\documentclass[11pt]{amsart}
\usepackage[utf8]{inputenc}
\usepackage{amssymb,amsmath}
\usepackage{graphicx}
\usepackage{color}
\usepackage{tikz}
\usepackage{pgfplots}
\pgfplotsset{compat=newest} 
\usepackage{enumitem}
\usepackage{algorithm}
\usepackage{setspace}
\usepackage{algpseudocode}
\usepackage{mathtools}
\usepackage{xfrac}

\usepackage{accents}
\usepackage{multicol} 
\usepackage{soul}
\usepackage[normalem]{ulem}
\usepackage{subcaption}
\usepackage{booktabs}
\usepackage{array}
\newcolumntype{L}[1]{>{\raggedright\let\newline\\\arraybackslash\hspace{0pt}}m{#1}}
\newcolumntype{C}[1]{>{\centering\let\newline\\\arraybackslash\hspace{0pt}}m{#1}}
\newcolumntype{R}[1]{>{\raggedleft\let\newline\\\arraybackslash\hspace{0pt}}m{#1}}
\usepackage{siunitx}
\usepackage{bbold}

\usepackage[hidelinks]{hyperref}
\usepackage[nameinlink,noabbrev]{cleveref}
\newtheorem{theorem}{Theorem}
\newtheorem{proposition}[theorem]{Proposition}

\theoremstyle{definition}

\newtheorem{lemma}[theorem]{Lemma}

\theoremstyle{remark}
\newtheorem{remark}[theorem]{Remark}

\Crefname{assumption}{Assumption}{Assumptions}
\numberwithin{theorem}{section}
\numberwithin{equation}{section}
\numberwithin{table}{section}
\numberwithin{figure}{section}
\definecolor{myBlue}{RGB}{30,144,255} 
\definecolor{myGreen}{RGB}{69,169,0} 
\definecolor{myRed}{RGB}{165,12,42} 
\definecolor{myOrange}{RGB}{225,92,22} 
\definecolor{color0}{rgb}{0.12156862745098,0.466666666666667,0.705882352941177}
\definecolor{color1}{rgb}{1,0.498039215686275,0.0549019607843137}
\definecolor{color2}{rgb}{0.172549019607843,0.627450980392157,0.172549019607843}
\definecolor{color3}{rgb}{0.83921568627451,0.152941176470588,0.156862745098039}
\definecolor{color4}{rgb}{0.580392156862745,0.403921568627451,0.741176470588235}
\definecolor{color5}{rgb}{0,0,0}

\newcommand{\delete}[1]{ }

\def\N{\mathbb{N}}

\def\R{\mathbb{R}}

\newcommand\ds{\,\text{d}s}

\newcommand\dx{\,\text{d}x}

\newcommand{\tddt}{\ensuremath{\tfrac{\text{d}}{\text{d}t}} }

\DeclareMathOperator{\id}{id}

\DeclareMathOperator{\trace}{trace}

\newcommand{\calA}{\ensuremath{\mathcal{A}} }

\newcommand{\calB}{\ensuremath{\mathcal{B}} }
\newcommand{\calBB}{\ensuremath{\widetilde{\mathcal{B}}} }
\newcommand{\calC}{\ensuremath{\mathcal{C}} }
\newcommand{\calCC}{\ensuremath{\widetilde{\mathcal{C}}} }
\newcommand{\calD}{\ensuremath{\mathcal{D}} }

\newcommand{\cHV}{\ensuremath{{\mathcal{H}_{\scalebox{.5}{\V}}}}}
\newcommand{\cHQ}{\ensuremath{{\mathcal{H}_{\scalebox{.5}{\Q}}}}}
\newcommand{\cHQdual}{\ensuremath{{\mathcal{H}^*_{\scalebox{.5}{\Q}}}}}

\newcommand{\calO}{\ensuremath{\mathcal{O}} }
\newcommand{\Q}{\ensuremath{\mathcal{Q}} }

\newcommand{\V}{\ensuremath{\mathcal{V}}}

\newcommand{\h}{\ensuremath{r}}

\newcommand{\shift}[2]{\Delta_{#1}{#2}}
\newcommand{\bdf}{\mathrm{BDF}_2}
\newcommand{\corr}[1]{{\color{color3}{#1}}} 

\allowdisplaybreaks
\begin{document}
\title[Semi-explicit integration of second order for poroelasticity]{Semi-explicit integration of second order\\ for weakly coupled poroelasticity} 
\author[]{R.~Altmann$^\dagger$, R.~Maier$^\ddagger$, B.~Unger$^\star$}
\address{${}^{\dagger}$ Institute of Analysis and Numerics, Otto von Guericke University Magdeburg, Universit\"atsplatz 2, 39106 Magdeburg, Germany}
\email{robert.altmann@ovgu.de}
\address{${}^{\ddagger}$ Institute for Applied and Numerical Mathematics, Karlsruhe Institute of Technology, Englerstr.~2, 76131 Karlsruhe, Germany}
\email{roland.maier@kit.edu}
\address{${}^{\star}$ Stuttgart Center for Simulation Science (SC SimTech), University of Stuttgart, Universit\"{a}tsstr.~32, 70569 Stuttgart, Germany}
\email{benjamin.unger@simtech.uni-stuttgart.de}
%
\date{\today}
%
%
\begin{abstract}
We introduce a semi-explicit time-stepping scheme of second order for linear poroelasticity satisfying a weak coupling condition. Here, semi-explicit means that the system, which needs to be solved in each step, decouples and hence improves the computational efficiency. The construction and the convergence proof are based on the connection to a differential equation with two time delays, namely one and two times the step size. Numerical experiments confirm the theoretical results and indicate the applicability to higher-order schemes. 
\end{abstract}
%
%
\maketitle
%
%
%
%
\section{Introduction}
This paper is devoted to the construction and analysis of a semi-explicit time discretization scheme of second order for linear poroelasticity~\cite{DetC93, Sho00}. The poroelastic equations can be characterized as a coupled system consisting of an elliptic and a parabolic equation and appear, e.g., in the field of geomechanics~\cite{Biot41, Zob10}. In many applications, this coupling is rather weak in a certain sense (cf.~\eqref{eq:coupling1} and~\eqref{eqn:weakCouplingCondition} below as well as the typical poroelastic parameters stated in~\cite[p.~25]{DetC93}), which is also a central assumption in this paper to guarantee convergence. 
%
For the temporal discretization of elliptic--parabolic problems such as poroelasticity, one mainly considers \emph{implicit} schemes such as the implicit Euler method~\cite{ErnM09} or higher-order schemes~\cite{Fu19}. This is primarily due to the fact that a semi-discretization in space yields a differential--algebraic equation for which \emph{explicit} time-stepping schemes cannot be used~\cite{KunM06}.  

Then again, one is interested in a decoupled approach, in which the elliptic and parabolic equations can be solved sequentially. Such a decoupling does not only replace the solution of a large system by two smaller subsystems to be solved but also enables the application of standard preconditioners~\cite{LMW17}. Moreover, the decoupling of the systems favors a co-design paradigm, allowing the usage of highly optimized software packages for the porous media flow (the parabolic equation) and the mechanical problem (the elliptic equation) separately, and, in addition, includes a linearization step if the permeability depends on the displacement, cf.~\cite{AltM22}.
One attempt in this direction are iterative decoupling methods such as the \emph{fixed-stress}, \emph{fixed-strain}, or \emph{drained} splitting schemes; see, e.g.,~\cite{ArmS92,WheG07,KimTJ11a,MikW13}. These schemes come along with an additional inner iteration in each time step that is required to guarantee convergence~\cite{StoBKNR19} and, additionally, require a careful selection of tuning parameters. In~\cite{ChaR18b}, an alternative method based on an additional stabilization term rather than an inner iteration is proposed. Although first-order convergence in time is observed in experiments, the theory presented in~\cite{ChaR18b,ChaR18} only guarantees suboptimal convergence of order $1/2$. Moreover, an extension to a higher-order method is by far not intuitive. Similarly, extensions of the aforementioned iterative schemes would require many additional inner iterations to guarantee the prescribed accuracy, counteracting the aim of an efficient numerical method. 

To combine the advantages of the monolithic and iterative coupling methods, a \emph{semi-explicit} time-stepping scheme was introduced in~\cite{AltMU21}, which decouples the equations and does not require an additional inner iteration or stabilization parameters. For a comparison of this method with the monolithic and different decoupling strategies, we refer to \cite{Muj22}. We emphasize that the semi-explicit scheme equals the implicit Euler discretization up to a term with a time shift in one of the equations. The perception of this scheme in terms of delay equations allows proving convergence of the method if a weak coupling condition is satisfied. This condition is independent of the step size and can be quantified explicitly. 

In this paper, we extend these ideas to construct and analyze a novel higher-order decoupling time integrator for coupled elliptic--parabolic problems, which include linear poroelasticity as a special case. To the best of our knowledge, this is the first time that a rigorous convergence analysis for a higher-order decoupling time discretization scheme is presented. For the construction of our scheme, we follow the general strategy developed in~\cite{AltMU21} and first construct a nearby delay system, which is then discretized in time. Recognizing that the first-order semi-explicit scheme analyzed in~\cite{AltMU21} can be understood as zeroth-order Taylor expansion, a straightforward approach would combine higher-order Taylor expansions (for the construction of the delay system) with higher-order time-integration schemes. The resulting delay equation, however, would be of advanced type, such that a sufficient regularity of the solution cannot be guaranteed as indicated in Section~\ref{sec:construction:Taylor}. Instead, we proceed with an expansion including multiple delays and use a backward differentiation formula~(BDF) for the time discretization of the resulting delay equation. Our main contributions are:
\begin{itemize}
	\item A BDF-type expansion to construct a delay equation that differs with a given order from the original elliptic--parabolic problem; cf.~Theorem~\ref{thm:multipleDelay}. This results in multiple delays in the first equation that enable the decoupling of the equations without the requirement for an inner iteration or additional tuning parameters.\\[-0.5em]
	\item A convergence proof for the second-order case in Theorems~\ref{thm:convergence} and~\ref{thm:BDF2}.
	For this, we solely work with the delayed parabolic equation that is obtained by resolving the elliptic equation and suitably adapt ideas from~\cite{AltMU21}. 	
	Moreover, we point out how this approach can be extended to higher orders. 
\end{itemize}
As in the first-order case, our method depends on a weak coupling condition, which we explicitly quantify via the theory of delay differential--algebraic equations in Section~\ref{sec:construction:stability}. We emphasize that the coupling strength of the two equations is also of relevance for the iterative decoupling methods mentioned earlier in the sense that they require more inner iterations if the coupling is stronger. Hence, they become inefficient for strongly coupled problems.  

Since we focus on time discretization, the whole convergence analysis is given on operator level, i.e., without a spatial discretization. Corresponding results for the fully discrete scheme can be obtained by the introduction of appropriate Ritz projections, cf.~\cite{AltCMPP20,AltMU21}. We conclude our presentation with three numerical examples in \Cref{sec:num}. 
\subsection*{Notation}  
We write~$a \lesssim b$ to indicate the existence of a generic constant~$C$, independent of spatial and temporal discretization parameters, such that~$a \leq C b$. 
%
%
\section{Poroelastic Equations}\label{sec:model}
In this section, we introduce the equations of linear poroelasticity and the corresponding abstract formulation as an elliptic--parabolic problem. 
We consider a bounded Lipschitz domain $\Omega \subseteq \R^d$, $d\in\{2,3\}$, in which we seek the displacement field~$u\colon[0,T] \times \Omega \rightarrow \R^d$ and the pore pressure~$p\colon[0,T] \times \Omega \rightarrow \R$. For a given time horizon $T>0$, the system equations read 
\begin{subequations}
\label{eq:poroStrong}
\begin{align}
	-\nabla \cdot \sigma(u) + \nabla (\alpha p) &\;=\; f\qquad \text{in } (0,T] \times \Omega, \\
	\partial_t\Big(\alpha \nabla \cdot u + \frac{1}{M} p \Big)- \nabla \cdot \Big( \frac{\kappa}{\nu}\, \nabla p \Big) &\;=\; g\qquad \text{in } (0,T] \times \Omega 
\end{align}
together with initial conditions  
\begin{align}
  u(0) = u^0, \qquad
  p(0) = p^0.
\end{align}
\end{subequations}
Therein, $\sigma$ denotes the stress tensor 
\[
  \sigma(u) 
  = \mu \, \big(\nabla u + (\nabla u)^T \big) + \lambda\, (\nabla \cdot u) \,\id
\]
with Lam\'e coefficients $\lambda$ and $\mu$, the permeability $\kappa$, the Biot-Willis fluid-solid coupling coefficient $\alpha$, the Biot modulus $M$, and the fluid viscosity~$\nu$; see~\cite{Biot41,Sho00}. Since some of these coefficients play a central role for the analysis of our scheme, we report the coefficients for a selection of different materials in \Cref{tab:examplesWeakCoupling}. The right-hand sides $f$ and $g$ are the volumetric load and the fluid source, respectively, modeling an injection or production process. 
Throughout this paper, we assume homogeneous Dirichlet boundary conditions, i.e., we set~$u = 0$ and~$p = 0$ on $(0,T] \times \partial \Omega$. 
%
%
\subsection{Abstract formulation}\label{sec:model:abstract}
For an abstract formulation of~\eqref{eq:poroStrong}, we introduce the Hilbert spaces 
\begin{equation*}
  \V \coloneqq [H^1_0(\Omega)]^d, \qquad
  \cHV \coloneqq [L^2(\Omega)]^d, \qquad 
  \Q \coloneqq H^1_0(\Omega), \qquad 
  \cHQ \coloneqq L^2(\Omega)
\end{equation*}
which include the assumed Dirichlet boundary conditions. With the respective dual spaces of~$\V$ and $\Q$ denoted by $\V^*$ and $\Q^*$, $(\V,\cHV,\V^*)$ as well as~$(\Q,\cHQ,\Q^*)$ form Gelfand triples with dense embeddings; see~\cite[Ch.~23.4]{Zei90a} for more details. Moreover, we define the bilinear forms 
\begin{align*}
  a(u,v) 
  \coloneqq \int_\Omega \sigma(u) : \varepsilon(v) \dx&, \qquad 
  b(p,q) 
  \coloneqq \int_\Omega \frac{\kappa}{\nu}\ \nabla p \cdot \nabla q \dx,\\
  c(p,q) 
  \coloneqq \int_\Omega \frac{1}{M}\ p\, q \dx&, \qquad 
  d(u,q) 
  \coloneqq \int_\Omega \alpha\, (\nabla \cdot u)\, q \dx
\end{align*}
with the classical double dot notation, i.e., for matrices $A,B\in\R^{n\times m}$ we have $A:B = \trace(A^T B)$, and the symmetric gradient~$\varepsilon(u) \vcentcolon= \tfrac{1}{2}(\nabla u + (\nabla u)^T)$ used in continuum mechanics. 
With this, the weak formulation of~\eqref{eq:poroStrong} can be written as follows: 
seek $u\colon[0,T] \to \V$ and~$p\colon[0,T] \to \Q$ such that 
\begin{subequations}
\label{eq:poroAbstract}
\begin{align}
    a(u,v) - d(v, p) 
    &= \langle f, v \rangle, \label{eq:poroAbstract:a} \\
    d(\dot u, q) + c(\dot p,q) + b(p,q) 
    &= \langle g, q\rangle \label{eq:poroAbstract:b} 
\end{align}
\end{subequations}
for all test functions~$v\in \V$, $q \in \Q$. Correspondingly, we assume that the right-hand sides satisfy~$f\colon[0,T] \to \V^*$ and $g\colon[0,T] \to \Q^*$ and denote with $\langle\cdot,\cdot\rangle$ the respective duality pairings. 
We would like to emphasize that it is sufficient to prescribe initial data for~$p$, since equation~\eqref{eq:poroAbstract:a} defines a consistency condition for $p^0$ and~$u^0$ (which is uniquely solvable for $u^0$; see the forthcoming discussion on the properties of the bilinear forms). 

System~\eqref{eq:poroAbstract} may also be written in operator form in the dual spaces of $\V$ and~$\Q$. For this, let $\calA$, $\calB$, $\calC$, and $\calD$ denote the operators corresponding to the bilinear forms $a$, $b$, $c$, and $d$, respectively. Then, \eqref{eq:poroAbstract} is equivalent to 
\begin{align*}
  \calA u - \calD^* p &= f\qquad\, \text{in }\V^*, \\
  \calD \dot u + \calC \dot p + \calB p &= g\qquad\; \text{in }\Q^*.
\end{align*}

It remains to discuss the properties of the bilinear forms. The bilinear form $a\colon \V\times\V\to \R$ is symmetric, elliptic, and bounded, i.e., there exist positive constants $c_a, C_a$ such that 
\begin{equation*}
  a(u,u) 
  \geq c_a\, \Vert u \Vert_\V^2, \qquad
  a(u,v) 
  \leq C_a\, \Vert u \Vert_\V \Vert v\Vert_\V
\end{equation*} 
for all~$u, v\in \V$. We would like to emphasize that $a$ is well known from the theory of linear elasticity and that the ellipticity follows from Korn's inequality~\cite[Th.~6.3.4]{Cia88}. 
%
Similarly, $b\colon \Q\times\Q\to \R$ is symmetric, elliptic, and bounded in~$\Q$, i.e., there exist positive constants $c_b, C_b$ such that  
\begin{equation*}
  b(p,p) 
  \geq c_b\, \Vert p \Vert_\Q^2, \qquad
  b(p,q)  
  \leq C_b\, \Vert p \Vert_\Q \Vert q\Vert_\Q
\end{equation*}
for all~$p, q\in \Q$. 
%
The bilinear form $c \colon \cHQ\times \cHQ\to\R$ simply involves the multiplication by a (positive) constant and, hence, defines an inner product in the pivot space~$\cHQ$. In more detail, there exist positive constants $c_c, C_c$ such that 
\begin{equation*}
  c(p,p) 
  \geq c_c\, \Vert p \Vert_{\cHQ}^2, \qquad
  c(p,q) 
  \leq C_c\, \Vert p \Vert_\cHQ \Vert q\Vert_\cHQ
\end{equation*}
for all~$p, q\in \cHQ$. 
%
The remaining bilinear form~$d\colon \V\times\cHQ\to \R$ models the coupling and is continuous, i.e., there exists a positive constant $C_d$ such that 
\begin{equation*}
  d(u,p) 
  \leq C_d\, \Vert u \Vert_\V \Vert p\Vert_\cHQ
\end{equation*}
for all~$u \in \V$ and $p\in \cHQ$. 
\begin{remark}
System~\eqref{eq:poroAbstract} can also be used to model linear thermoelasticity, which considers the displacement of a material due to temperature changes~\cite{Bio56}. 	
More generally, system~\eqref{eq:poroAbstract} is an elliptic--parabolic system, where the elliptic part (modeled by $a$) and the parabolic part (modeled by $b$ and $c$) are coupled through the bilinear form~$d$. We emphasize that the forthcoming analysis does not depend on the specific application, but only on the properties of the bilinear forms introduced above. 
\end{remark}
%
%
\subsection{Spatial discretization}\label{sec:model:space}
Although this paper is mainly concerned with the temporal discretization, we shortly comment on the finite element discretization of~\eqref{eq:poroAbstract}. For more details, we refer to~\cite{ErnM09}. 
In order to transfer the convergence results of this paper to the fully discrete system, one may consider spatial projection operators corresponding to the elliptic bilinear forms~$a$ and~$b$; see~\cite{AltMU21}. 

Considering finite-dimensional subspaces $V_h\subseteq \V$ and $Q_h\subseteq \Q$, one seeks approximations~$u_h \approx u$ and~$p_h \approx p$. Here, the parameter~$h$ represents the mesh size of the triangulation used in the construction of $V_h$ and $Q_h$. A direct spatial discretization of~\eqref{eq:poroAbstract} then leads to the differential--algebraic equation  
\[
\begin{bmatrix} 0 & 0 \\ D & M_c \end{bmatrix}
\begin{bmatrix} \dot u_h \\ \dot p_h \end{bmatrix}
=  
\begin{bmatrix} -K_a & D^T \\ 0 & -K_b \end{bmatrix}
\begin{bmatrix} u_h \\ p_h \end{bmatrix}
+ \begin{bmatrix} f_h \\ g_h \end{bmatrix}.
\]
Therein, $K_a$ and~$K_b$ denote the stiffness matrices corresponding to the bilinear forms $a$ and~$b$, respectively. Due to the assumptions discussed above, $K_a$ and~$K_b$ can be assumed to be symmetric and positive definite. Moreover, $M_c$ equals the mass matrix corresponding to~$c$, which is thus also symmetric and positive definite, and $D$ is a rectangular matrix corresponding to~$d$. 

Using standard $P_1$ finite elements to define $V_h$ and $Q_h$, one obtains the expected convergence rates of order one in the energy norms and order two in the $L^2$-norms. For more precise results also on higher-order approximations, we again refer to~\cite{ErnM09}. 
%
%
\subsection{Temporal discretization of first order}\label{sec:model:time}
The standard way to discretize system~\eqref{eq:poroAbstract} in time is the application of the implicit Euler scheme. This results in a time-stepping scheme of order one as shown in~\cite{ErnM09}. 

As already mentioned in the introduction, the differential--algebraic structure rules out the possibility of a fully explicit discretization in time. In~\cite{AltMU21}, however, a semi-explicit scheme was introduced. Considering an equidistant decomposition of $[0,T]$ with step size~$\tau$, this scheme reads  
\begin{subequations}
\label{eqn:semiExpl:Ord1}
\begin{align}
  a(u^{n+1},v) - d(v, p^{n}) 
  &= \langle f^{n+1}, v \rangle, \\
  \tfrac 1\tau\, d(u^{n+1}-u^n, q) + \tfrac 1\tau\, c(p^{n+1}-p^n,q) + b(p^{n+1},q) 
  &= \langle g^{n+1}, q\rangle 
\end{align}
\end{subequations}
for all~$v\in \V$, $q \in \Q$. Here, $u^n$ and $p^n$ denote the approximations of $u(t^n)$ and $p(t^n)$, $t^n=n\tau$, respectively. Note that, in contrast to the implicit Euler discretization, the first equation contains $p^n$ rather than $p^{n+1}$. Hence, the two equations decouple and can be solved sequentially. It is shown in~\cite{AltMU21} that this maintains the first-order convergence as long as the \emph{weak coupling condition}
\begin{equation}\label{eq:coupling1}
	\alpha^2 M 
	\le \mu+\lambda 
\end{equation}
is satisfied. Note that this condition is specific to the equations of poroelasticity. For general elliptic--parabolic problems defined via the bilinear forms~$a$, $b$, $c$, and~$d$, the weak coupling conditions reads~$C_d^2 \le c_a c_c$. 

For the convergence analysis, the connection of system~\eqref{eq:poroAbstract} to a related delay system is used. This idea is also applied in the following sections to construct a semi-explicit scheme of order two, leading to a slightly more restrictive weak coupling condition. 
%
%
\section{Semi-explicit Integration Scheme of Second Order}\label{sec:construction}
This section is devoted to the extension of the semi-explicit scheme~\eqref{eqn:semiExpl:Ord1} to second order. Following the idea in \cite{AltMU21}, we first construct a related delay equation and then discretize the delay equation with an implicit scheme of second order. For this, we need to replace the pressure in the first poroelastic equation by a time-delayed term which is second-order accurate. We first consider a Taylor expansion before we introduce discrete derivatives, leading to a system with multiple delays. 

In the following, we consider a uniform partition of the time interval~$[0,T]$ with step size~$\tau>0$ such that~$N \coloneqq T/\tau\in\N$. Hence, we consider time points~$t^n=n\tau$ for~$n=0,\dots,N$. The approximation of a function $y$ at time $t^n$ is then denoted by~$y^n$. 
%
%
\subsection{Related delay systems by Taylor expansion}\label{sec:construction:Taylor}
In a first step, we aim to decouple the elliptic and parabolic equation in~\eqref{eq:poroAbstract} by replacing~$p$ in the first equation by a Taylor expansion. For a first-order approximation, one can simply replace $p(t)$ by the delay term~$\shift{\tau}{p(t)} \coloneqq p(t-\tau)$, cf.~\Cref{sec:model:time}. In the general case, we replace $p(t)$ by the Taylor expansion of order $k$ at time $t-\tau$. This then leads to the delay system 
\begin{subequations}
\label{eqn:delay:Taylor}
\begin{align}
  a(\bar{u},v) - d\big(v, {\textstyle\sum}_{j=0}^{k-1} \tfrac{\tau^{j}}{j!}\, \shift{\tau}{\bar{p}^{(j)}} \big) 
  &= \langle f, v \rangle, \label{eqn:delay:Taylor:a} \\
  d(\dot {\bar{u}}, q) + c(\dot {\bar{p}},q) + b(\bar{p},q) 
  &= \langle g, q\rangle \label{eqn:delay:Taylor:b}
\end{align}
\end{subequations}
for test functions~$v\in \V$ and~$q \in \Q$. As initial condition, we set~$\bar{p}(0)=p(0)=p^0$. In contrast to the original system, however, we need additional information on~$\bar{p}$ (and its derivatives) on the interval~$[-\tau, 0]$.  
For this, we introduce a so-called {\em history function} $\bar{\Phi}(t) = \bar{p}\big|_{[-\tau, 0]}(t)$, for which we assume~$\bar{\Phi} \in C^\infty([-\tau, 0], \Q)$ with 
\begin{equation}
\label{eqn:historyTaylor}
	\bar{\Phi}(-\tau) = \bar{\Phi}(0) = p^0, \qquad
	\bar{\Phi}^{(j)}(-\tau) = 0\quad \text{for } j=1, \dots, k-1.
\end{equation}
The conditions on the derivatives of~$\bar{\Phi}$ ensure the consistency of the initial data, i.e., $\bar{u}(0)=u(0)=u^0$. 

Since system~\eqref{eqn:delay:Taylor} is constructed by the help of a Taylor expansion, it is no surprise that the solutions~$(u,p)$ and~$(\bar{u},\bar{p})$ only differ by a term of order~$\tau^k$ as long as the solution of the delay systems stays stable. We refer to \Cref{appendix} for further details.
Nevertheless, system~\eqref{eqn:delay:Taylor} is not well-suited for the construction of a numerical scheme. This is due to the appearance of temporal derivatives, which turn out to be critical as we illustrate in the sequel. 
%
Already in the case of interest, namely $k=2$, the resulting delay system~\eqref{eqn:delay:Taylor} is of \emph{advanced} type \cite{BelC63} and, hence, only well-posed in a distributional setting \cite{TreU19}. This can be seen as follows. In operator form, the delay equation reads
\begin{equation*}
 \calA \bar{u} - \calD^*(\shift{\tau}{\bar{p}} + \tau\shift{\tau}{\dot{\bar{p}}})
 = f, \qquad 
 \calD \dot{\bar{u}} + \calC \dot{\bar{p}} + \calB \bar{p} 
  = g.
\end{equation*}
The first equation yields $\bar u = \calA^{-1} \calD^* (\shift{\tau}{\bar{p}} + \tau\, \shift{\tau}{\dot{\bar{p}}} + f)$. Inserting this in the second equation, we obtain 
\begin{equation*}
  \calC \dot{\bar{p}} + \calB \bar{p} 
  = g - \calD \dot{\bar{u}}
  = g - \calD \calA^{-1} \calD^* (\shift{\tau}{\dot{\bar{p}}} + \tau\, \shift{\tau}{\ddot{\bar{p}}} + \dot{f}). 
\end{equation*}
For a given history function, this is a parabolic problem that can be solved on the interval $(0,\tau)$. We thus can proceed iteratively to construct a solution. Nevertheless, due to the term $\shift{\tau}{\ddot{\bar{p}}}$, this solution loses regularity over time, and is hence not suitable for numerical methods. We refer to \cite{BelZ03,AltZ18,Ung18} for further details. 

To avoid such advanced delay systems, we consider an alternative approach and replace the derivatives by discrete derivatives. This then yields a system with multiple delays.  
%
%
\subsection{Discrete derivatives}\label{sec:construction:discreteDerivatives}
Based on the discrete difference operator $D y^{n+1} \vcentcolon= y^{n+1} - y^{n}$, we can write discrete derivatives of any order in a short way. 
As an example, the difference quotients of order one and two read
\[
\frac1\tau D y^{n+1} 
= \frac{y^{n+1} - y^{n}}{\tau},\qquad 
\frac1\tau \big(D y^{n+2} + \tfrac12 D^2 y^{n+2}\big)
= \frac{3y^{n+2}-4y^{n+1}+y^n}{2\tau},
\]
leading, e.g., to the well-known BDF schemes. For the upcoming analysis, it is convenient to extend the definition of the difference operator also to continuous functions. More precisely, we define~$D y \vcentcolon= y - \shift{\tau}{y}$ with the time shift~$\shift{\tau}{}$ introduced above. The resulting order-$k$ approximations of the derivative of a function are summarized in the following lemma. 
\begin{lemma}[discrete derivative]
\label{lem:BDF}
Let $y\in C^{k}([0,T])$ and $t\in [k\tau,T]$. Then, it holds that 
\begin{equation*}
	\dot y(t)
	= \frac1\tau\, \sum_{j=1}^{k} \tfrac1j\, D^j\, y(t) + \calO(\tau^k).
\end{equation*}
Moreover, if we have $y\in C^{k+1}([0,T])$, then there exist constants $c_1,\ldots,c_k$ such that the error term can be written as $\sum_{j=1}^k c_j \int_{t-j\tau}^t (t-\xi)^j y^{(j+1)}(\xi)\,\mathrm{d}\xi$. 
\end{lemma}
%
%
\subsection{Related delay system with multiple delays}\label{sec:construction:multDelay}
As already mentioned, we want to replace the derivatives in~\eqref{eqn:delay:Taylor} by discrete derivatives. Focusing on the case $k=2$, which will lead to a scheme of second order, we replace $\tau \shift{\tau}{\dot{\bar{p}}}$ by $D  \shift{\tau}{\bar{p}} = \shift{\tau}{\bar{p}} - \shift{2\tau}{\bar{p}}$. This leads to the system 
\begin{subequations}
\label{eqn:delay:multipleDelay}
\begin{align}
  a(\tilde{u},v) - d\big(v, 2 \shift{\tau}{\tilde{p}} - \shift{2\tau}{\tilde{p}} \big) 
  &= \langle f, v \rangle, \label{eqn:delay:multipleDelay:a} \\
  d(\dot {\tilde{u}}, q) + c(\dot {\tilde{p}},q) + b(\tilde{p},q) 
  &= \langle g, q\rangle \label{eqn:delay:multipleDelay:b}
\end{align}
\end{subequations}
for all test functions~$v\in \V$ and~$q \in \Q$. Note that this is a system with two delays, namely $\tau$ and $2\tau$. 
%
Again, we need to discuss the initial data, which includes $\tilde{p}(0)=p(0)=p^0$ and an appropriate history function~$\tilde{\Phi}$ defined on $[-2\tau,0]$. To obtain consistency in $u$, namely~$\tilde{u}(0)=u(0)=u^0$, we assume that~$\tilde{\Phi} \in C^\infty([-2\tau, 0], \Q)$ satisfies  
\begin{align}
\label{eqn:historyMultipleDelays}
p^0
= 2 \shift{\tau}{\tilde{p}}(0) - \shift{2\tau}{\tilde{p}}(0)
= 2\, \tilde{\Phi}(-\tau) - \tilde{\Phi}(-2\tau).
\end{align}
%
%
\begin{remark}[approximations of higher order]
\label{rem:higherOrderBDF}
For general $k\ge1$, one possibility is to replace the derivatives $\shift{\tau}{\bar{p}^{(j)}}$ in~\eqref{eqn:delay:Taylor} by approximations of order $\tau^{k-j}$. 
This then guarantees that the resulting expression is an approximation of the Taylor expansion ${\textstyle\sum}_{j=0}^{k-1} \tfrac{\tau^{j}}{j!}\, \shift{\tau}{\bar{p}^{(j)}}$ of order $k$. 
Note, however, that this leads to a growing number of delays. For $k=3$ this yields three delays, whereas $k=4$ already needs five delays. The resulting scheme for $k=3$ is presented in \Cref{sec:num:order3}.
\end{remark} 
\begin{remark}[parabolic equation with multiple delays]
Considering the operator formulation of~\eqref{eqn:delay:multipleDelay} and eliminating the variable~$\tilde{u}$ by the first equation, we get 
\begin{align}
\label{eq:parabolicDelaySect3}
  \calC \dot{\tilde{p}} + \calB \tilde{p} + \calD\calA^{-1}\calD^* (2\shift{\tau}{\dot{\tilde{p}}} - \shift{2\tau}{\dot{\tilde{p}}})
  = g - \calD\calA^{-1} \dot{f}.
\end{align}
Note that this is a parabolic equation (of neutral type) with two delays. Hence, we consider here multiple delays rather than higher derivatives. 
\end{remark}
Motivated by the approximation properties of the Taylor expansion approach, the following theorem shows that the solutions to~\eqref{eq:poroAbstract} and~\eqref{eqn:delay:multipleDelay} only differ by a term of order two. 
\begin{theorem}
\label{thm:multipleDelay}
Assume sufficiently smooth right-hand sides~$f$ and~$g$ and a history function~$\tilde{\Phi}$ satisfying~\eqref{eqn:historyMultipleDelays}. 
Then, the solutions to~\eqref{eq:poroAbstract} and~\eqref{eqn:delay:multipleDelay} are equal up to a term of order~$\tau^2$, i.e., for almost all~$t\in[0,T]$ we have 
\begin{equation*}
	\Vert \tilde{p}(t) - p(t) \Vert^2_\Q + \Vert \tilde{u}(t) - u(t)\Vert^2_\V 
	\ \lesssim\ t\, \tau^{4}.
\end{equation*}
Here, the hidden constant depends on higher derivatives of the history function~$\tilde{\Phi}$ as well as of $\tilde{p}$. 
\end{theorem}
\begin{proof}
We define~$e_p \coloneqq \tilde{p} - p$ and~$e_u \coloneqq \tilde{u} - u$. Due to the assumptions on the history function, we conclude that~$e_p(0)=0$ and~$e_u(0)=0$. Considering the difference of~\eqref{eqn:delay:multipleDelay:a} and~\eqref{eq:poroAbstract:a}, we obtain 
\begin{align*}
  a(e_u,v) - d(v, e_p) 
  = - d(v, \tilde{p} - 2\shift{\tau}{\tilde{p}} + \shift{2\tau}{\tilde{p}} )
  \le \tau^2\, C_d\, \|v\|_\V \|\ddot{\tilde{p}}\|_{L^\infty(-2\tau,T;\cHQ)},
\end{align*}
where $L^\infty(-2\tau,T;\cHQ)$ denotes the Bochner space on the time interval~$(-2\tau,T)$ with values in $\cHQ$. In the same manner, we obtain by the derivatives of~\eqref{eqn:delay:multipleDelay:a} and~\eqref{eq:poroAbstract:a} that  
\begin{align*}
  a(\dot e_u,v) - d(v, \dot e_p) 
  \le \tau^2\, C_d\, \|v\|_\V \|\tilde{p}^{(3)}\|_{L^\infty(-2\tau,T;\cHQ)}. 
\end{align*}
Now we can proceed as in the proof of \Cref{prop:delayTaylor}, i.e., we consider the test function $v=\dot{e}_u$ in combination with the difference of~\eqref{eqn:delay:multipleDelay:b} and~\eqref{eq:poroAbstract:b}.	
\end{proof}
\begin{remark}
	The hidden constant in \Cref{thm:multipleDelay} may become arbitrarily large depending on the ellipticity and continuity constants. This is discussed in more detail in \Cref{sec:construction:stability}.
\end{remark}
System~\eqref{eqn:delay:multipleDelay} yields a good starting point for the construction of higher-order discretization schemes. This is subject of the following subsection. 
%
\subsection{Semi-explicit integration scheme}\label{sec:construction:scheme}
In order to obtain a semi-explicit time-stepping scheme, we now apply the BDF-$2$ scheme to~\eqref{eqn:delay:multipleDelay}. To shorten notation, we introduce 
\[
  \bdf u^{n+2}
  \coloneqq 3 u^{n+2} - 4 u^{n+1} + u^{n}
  = 2D u^{n+2} + D^2 u^{n+2}.
\]
By \Cref{lem:BDF}, we know that $\frac{1}{2\tau}\bdf u(t) = \dot u(t) + \calO(\tau^2)$. Since the first equation does not contain any derivatives, the temporal discretization is simply given by a function evaluation at time $t^{n+2}$ (as for the implicit Euler scheme). This discretization yields the semi-explicit scheme 
\begin{subequations}
\label{eqn:semiExplicit:order2}
\begin{align}
  a(u^{n+2}, v) - d\big(v, 2 p^{n+1} - p^{n} \big) 
  &= \langle f^{n+2}, v\rangle, \label{eqn:semiExplicit:order2:a}\\
  \tfrac1{2\tau} d\big( \bdf u^{n+2}, q\big) + \tfrac1{2\tau} c\big( \bdf\, p^{n+2}, q\big) + b(p^{n+2}, q) 
  &= \langle g^{n+2}, q\rangle \label{eqn:semiExplicit:order2:b}
\end{align}
\end{subequations}
%
%
for test functions~$v\in \V$ and~$q \in \Q$. Note that this is a~$2$-step scheme, calling for initial data~$p^0=p(0)$ and $p^1$. In place of the history function, we set~$p^{-2}, p^{-1} \in \Q$ such that 
\begin{equation}
\label{eq:pInitial}
  p^0 = 2\, p^{-1} - p^{-2}, \qquad
  p^1 = 2\, p^0 - p^{-1}. 
\end{equation}
The first condition corresponds to~\eqref{eqn:historyMultipleDelays} and gives the consistency condition for~$u^0$. The second equation ensures that $p^1$ and $u^1$ are consistent. To be precise, this means that the resulting values~$u^0, u^1\in\V$ satisfy 
\begin{equation}\label{eq:uInitial}
  a(u^0, v) - d\big(v, p^0 \big) = \langle f^0, v\rangle, \qquad
  a(u^1, v) - d\big(v, p^1 \big) = \langle f^1, v\rangle  
\end{equation}
for all $v\in \V$. 

The proposed scheme~\eqref{eqn:semiExplicit:order2} is indeed semi-explicit, since the first equation defines~$u^{n+2}$ purely by already computed values, i.e., without the knowledge of~$p^{n+2}$. Inserting this value in the second equation, we then obtain the approximation~$p^{n+2}$. 
In operator form, this scheme reads 
\begin{align*}
  \calA u^{n+2} - \calD^* \big(2p^{n+1}-p^n\big) 
  &= f^{n+2}, \\
  \calD \big(3u^{n+2} - 4u^{n+1} + u^n\big) + \calC \big(3p^{n+2} - 4p^{n+1} + p^n \big) + 2\tau \calB p^{n+2} 
  &= 2\tau g^{n+2}.
\end{align*}
Using once more the invertibility of the operator $\calA$,  
we can eliminate the $u$-variables in the second equation, leading to 
\begin{equation}\label{eq:discform}
\begin{aligned}
  \calC \big(3p^{n+2} &- 4p^{n+1} + p^n \big) 
  + 2\tau \calB p^{n+2}\\&
  + \calD\calA^{-1}\calD^* \big( 6p^{n+1} - 11p^n + 6p^{n-1} - p^{n-2} \big) \\
  &\hspace{3.5cm}= 2\tau g^{n+2}
  - \calD\calA^{-1} \big(3f^{n+2} - 4f^{n+1} + f^n \big).
\end{aligned}
\end{equation}
We would like to emphasize that this equals the BDF-$2$ discretization of the delay equation~\eqref{eq:parabolicDelaySect3}. 
This fact will be used in the following convergence result. 
\begin{theorem}[Second-order convergence of the semi-explicit scheme]
\label{thm:convergence}
Assume sufficiently smooth right-hand sides~$f$ and~$g$. Moreover, let the operators satisfy the {\em weak coupling condition} 
\begin{equation}
\label{eqn:weakCouplingCondition}
  \omega
  \coloneqq 
  \frac{\alpha^2 M}{\mu+\lambda}
  \le \frac 15.  
\end{equation}
Then the semi-explicit scheme~\eqref{eqn:semiExplicit:order2} converges with order two. More precisely, given $p^0=p(0)$ and~$p^1$ as a second-order approximation of~$p(\tau)$, 
we can define consistent $u^0$ and~$u^1$ in the sense of~\eqref{eq:uInitial} such that 
\[
  \|u(t^{n})-u^{n}\|^2_{\V}
  + \|p(t^n)-p^n\|^2_{\cHQ} + \tau \sum_{j=1}^n \|p(t^j)-p^j\|^2_\Q
  \lesssim t^n \tau^4 
\]
for all $n\ge 0$.  
\end{theorem}
\begin{proof}
Given $p^0$ and~$p^1$, we define~$p^{-2}$ and~$p^{-1}$ satisfying~\eqref{eq:pInitial} such that $u^0,u^1$ are consistent. Moreover, let~$\tilde{\Phi}$ be a history function with~$\tilde{\Phi}(-2\tau) = p^{-2}$ and $\tilde{\Phi}(-\tau) = p^{-1}$ such that~\eqref{eqn:historyMultipleDelays} is satisfied. 
We can now apply \Cref{thm:multipleDelay} and conclude that the exact solution and the solution of the delay system~\eqref{eqn:delay:multipleDelay} only differ by a term of order two. Hence, it is sufficient to compare the discrete solution given by~\eqref{eqn:semiExplicit:order2} with $(\tilde{p}, \tilde{u})$. 

We have seen that the presented semi-explicit scheme corresponds to the BDF-$2$ method applied to the delay equation~\eqref{eq:parabolicDelaySect3}. Since the operator $\calC$ only contains a multiplicative factor, we may consider a simple rescaling leading to the question of the convergence of the BDF-$2$ scheme applied to the delay system 
\begin{equation}\label{eq:delaypdep}
  \dot{\tilde{p}} + \calBB \tilde{p} + \calCC\, \big( 2\shift{\tau}{\dot{\tilde{p}}} - \shift{2\tau}{\dot{\tilde{p}}} \big)
  = \h 
  \coloneqq \calC^{-1} g - \calC^{-1}\calD\calA^{-1} \dot{f} 
\end{equation}	
with $\calBB\coloneqq \calC^{-1} \calB \colon \V \to \V^*$ and $\calCC\coloneqq\calC^{-1}\calD\calA^{-1}\calD^*\colon \cHQ \to \cHQdual$. Note that these two operators are symmetric, elliptic, and continuous in the respective spaces and that the continuity constant of $\calCC$ equals~$\omega$. 
In \Cref{thm:BDF2} of the following section, we show that this implies an estimate of the form 
%
\[
  \|\tilde{p}(t^n)-p^n\|^2_{\cHQ} + \tau \sum_{j=1}^n \|\tilde{p}(t^j)-p^j\|^2_\Q
  \lesssim t^n\, \tau^4 + t^n\, E_\mathrm{rhs} + E_\mathrm{init}
\]
for $n \geq 2$. 
The right-hand side error $E_\mathrm{rhs}$ appears because the approximation of the right-hand side in~\eqref{eq:delaypdep} involves a BDF-2 approximation of $f$ rather than the nodal evaluation; see~\eqref{eq:discform}. However, due to Lemma~\ref{lem:BDF}, $E_\mathrm{rhs}$ is of order $\calO(\tau^4)$. 
Due to the assumption on~$p^1$, we have $E_\mathrm{init} = \| \tilde p(\tau)-p^{1}\|^2 + \tau\, \| \tilde p(\tau)-p^{1}\|_{\tilde b}^2 \lesssim \tau^4$ and leads to an overall error of order two. Hence, the above estimate holds for all~$n \geq 0$. 
Moreover, considering the difference of equations~\eqref{eqn:delay:multipleDelay:a} and~\eqref{eqn:semiExplicit:order2:a}, we get by the ellipticity of the bilinear form $a$ that
\[
  \|\tilde{u}(t^{n+2})-u^{n+2}\|_{\V}
  \lesssim 2\, \|\tilde{p}(t^{n+1})-p^{n+1}\|_{\cHQ} + \|\tilde{p}(t^{n})-p^{n}\|_{\cHQ} 
\]
for $n \geq 0$. Finally, due to the consistency conditions for $u^0$ and $u^1$, we further get~$u(0)=\tilde u(0)=u^0$ and  
\begin{equation*}
  \|{u}(\tau)-u^{1}\|_{\V} 
  \lesssim \|{p}(\tau)-p^{1}\|_{\cHQ}.
\end{equation*}
The combination of the previous estimates completes the proof. 
\end{proof}
\begin{remark}[Initial data]
\label{rem:initDataImplEuler}	
In practice, appropriate initial conditions can be realized as follows: given $p^0$, one first computes~$u^0$ consistent to~\eqref{eq:poroAbstract:a}. Then, $p^1$ and $u^1$ can be obtained by a single step of the implicit Euler discretization applied to~\eqref{eq:poroAbstract}. This then guarantees consistency as well as the needed accuracy for $p^1$. 
\end{remark}
Before we discuss the convergence of the semi-explicit scheme, we focus on the weak coupling condition~\eqref{eqn:weakCouplingCondition} and its meaning in terms of delay equations. 
%
%
\subsection{Weak coupling condition and asymptotic stability of the delay system}
\label{sec:construction:stability}
First, let us emphasize that there are several poroelasticity problems reported in the literature that satisfy the weak coupling condition~\eqref{eqn:weakCouplingCondition}, or almost satisfy the weak coupling condition; see \Cref{tab:examplesWeakCoupling}. The latter will be relevant as well as the following discussion demonstrates. 

\begin{table}
	\centering
	\caption{Poroelasticity problems reported in the literature and their relation to the weak coupling condition~\eqref{eqn:weakCouplingCondition}. 
	All examples consider porous media in combination with water~\cite[Tab.~4]{DetC93}. }
	\label{tab:examplesWeakCoupling}
	\begin{tabular}{l||cc|cc|c|c}
		porous media & $\lambda$ & $\mu$ & $\alpha$ & $M$ & $\kappa / \nu$ & $\omega$\\\midrule\midrule
		%
		%
		Tennessee marble & $2.40\cdot10^{10}$ & $2.4\cdot10^{10}$ & 0.19 & $1.16\cdot10^{11}$ & $1.0\cdot10^{-19}$ & 0.09\\
		Charcoal granite & $2.23\cdot10^{10}$ & $1.9\cdot10^{10}$ & 0.27 & $8.50\cdot10^{10}$ & $1.0\cdot10^{-19}$ & 0.15 \\\midrule
		Weber sandstone & $5.14\cdot10^{9}$ & $1.2\cdot10^{10}$ & 0.64 & $2.79\cdot10^{10}$ & $1.0\cdot10^{-15}$ & 0.45\\ 
		Westerly granite & $1.5\cdot10^{10}$ & $1.5\cdot10^{10}$ & 0.47 & $7.64\cdot10^{10}$ & $4.0\cdot10^{-19}$ & 0.56\\\midrule		
		Berea sandstone & $4.00\cdot10^{9}$ & $6.0\cdot10^{9}$ & 0.79 & $1.23\cdot10^{10}$ & $1.9\cdot10^{-13}$ & 0.76\\
		Ruhr sandstone & $4.11\cdot10^{9}$ & $1.3\cdot10^{10}$ & 0.65 & $4.05\cdot10^{10}$ & $2.0\cdot10^{-16}$ & 1.00\\

	\end{tabular}
\end{table}

To see that the weak coupling condition is not a mere technical assumption, we analyze the asymptotic stability of the related delay system constructed in \Cref{sec:construction:multDelay} with multiple delays. To simplify the presentation, we consider here the finite-dimensional case after a semi-discretization in space (cf.~\Cref{sec:model:space}), and study the neutral delay differential equation corresponding to~\eqref{eq:parabolicDelaySect3}, i.e., we study the neutral delay equation
\begin{equation}
	\label{eqn:neutralDDE}
	\dot{\tilde{p}}_h + M_c^{-1}DK_a^{-1}D^T \left(2\shift{\tau}{\dot{\tilde{p}}_h} - \shift{2\tau}{\dot{\tilde{p}}_h}\right) + M_c^{-1}K_b\tilde{p}_h = \tilde{g}_h.
\end{equation}
A necessary condition (cf.~\cite[Thm.~3.20]{GuKC03}) for the delay-independent asymptotic stability of the unforced (i.e.,~$\tilde{g}_h=0$) delay equation~\eqref{eqn:neutralDDE} is that the spectral radius of the matrix
\begin{displaymath}
	N_2 \vcentcolon= \begin{bmatrix}
		-2M_c^{-1}DK_a^{-1}D^T & M_c^{-1}DK_a^{-1}D^T\\
		I & 0
	\end{bmatrix}
\end{displaymath}
is strictly less than one, i.e., $\rho(N_2) < 1$. Hereby, $I$ denotes the identity matrix of suitable dimension. We thus have to compute the eigenvalues of $N_2$.
Since $M_c$ is symmetric and positive definite, the (principle) square root $M_c^{1/2}$ exists and is symmetric and positive definite. Thus, the matrix $M_c^{-1/2}DK_a^{-1}D^TM_c^{-1/2}$ is symmetric and hence diagonalizable, i.e., there exists a diagonal matrix $\Lambda$ and an orthogonal matrix~$U$ such that 
\begin{displaymath}
	UM_c^{-1/2}DK_a^{-1}D^TM_c^{-1/2}U^{-1} = \Lambda.
\end{displaymath}
Define $\Xi \vcentcolon= \mathrm{diag}(UM_c^{1/2},UM_c^{1/2})$. Then, 
\begin{displaymath}
	\Xi N_2\Xi^{-1} = \begin{bmatrix}
		-2\Lambda & \Lambda\\
		I & 0
	\end{bmatrix}.
\end{displaymath}
For any eigenvalue $\lambda_\Lambda$ of $\Lambda$, it thus suffices to compute the spectral radius of the matrix
\begin{displaymath}
	\begin{bmatrix}
		-2\lambda_\Lambda & \lambda_\Lambda\\
		1 & 0
	\end{bmatrix},
\end{displaymath}
which is given by $\lambda_\Lambda + \sqrt{\lambda_\Lambda^2 + \lambda_\Lambda}$. Since this is a monotone expression, we conclude
\begin{displaymath}
	\rho(N_2) = \rho(M_c^{-1}DK_a^{-1}D^T) + \sqrt{\rho(M_c^{-1}DK_a^{-1}D^T)^2 + \rho(M_c^{-1}DK_a^{-1}D^T )},
\end{displaymath}
and thus $\rho(N_2) < 1$ if and only if $\rho(M_c^{-1}DK_a^{-1}D^T) < \tfrac{1}{3}$. Consequently, we cannot expect the delay equation to be a reasonable approximation of the non-delay equation if $\rho(M_c^{-1}DK_a^{-1}D^T) > \tfrac{1}{3}$. In fact, in the scalar case, it is easy to see that $\rho(M_c^{-1}DK_a^{-1}D^T) < \tfrac{1}{3}$ is also a sufficient condition for delay-independent asymptotic stability.
Using
\begin{displaymath}
	\frac{\alpha^2 M}{\mu+\lambda} 
	\leq \rho(M_c^{-1}DK_a^{-1}D^T) 
	< \tfrac{1}{3},
\end{displaymath}
we observe that a weak coupling condition as in~\eqref{eqn:weakCouplingCondition} is not only a technical requirement, but indeed necessary for convergence. We discuss the details in the error analysis in the next section.

\section{Convergence Analysis}\label{sec:errorAnalysis}
In this section, we prove the convergence of the BDF-$2$ method applied to the delay operator equation 
\begin{align}
\label{eq:parabolicDelay}
  \dot{z} + \calBB z + \calCC\, \big( 2\shift{\tau}{\dot{z}} - \shift{2\tau}{\dot{z}} \big)
  = \h.
\end{align}
Here, $\calBB\colon \Q\to\Q^*$ is an operator with the same properties as $\calB$ in the previous section and $\calCC\colon \cHQ\to\cHQdual$ is an operator with the same properties as $\calC$ with continuity constant~$\omega$. 
Similar as in \Cref{sec:construction:multDelay}, we assume, besides the initial condition $z(0)=z^0$, a given history function~$\Phi \in C^\infty([-2\tau, 0], \Q)$ with $\Phi(0) = z^0$. Moreover, the right-hand side~$\h\colon [0,T] \to \Q^*$ is sufficiently smooth. 

For the error analysis, we first require the following lemma. 
\begin{lemma}
\label{lem:symBilinear:k2}
For a symmetric bilinear form $\mathfrak{a}$ it holds that
\begin{align*}
  2\, \mathfrak{a}(z^{n+2}, \bdf z^{n+2})
  = \bdf\, \Vert z^{n+2} \Vert^2_\mathfrak{a}
  + 2\, \Vert D z^{n+2} \Vert^2_\mathfrak{a} 
  - 2\, \Vert D z^{n+1} \Vert^2_\mathfrak{a}
  + \Vert D^2 z^{n+2} \Vert^2_\mathfrak{a} 
\end{align*}
with $\Vert \cdot \Vert_\mathfrak{a}^2 \coloneqq \mathfrak{a}(\,\cdot\,,\cdot\,)$.
\end{lemma}
\begin{proof}
Using multiple applications of the formula 
\begin{align}
\label{eq:symBilinearForm}
  2\,\mathfrak{a}(x, x-y) 
  = \Vert x \Vert^2_\mathfrak{a} - \Vert y \Vert^2_\mathfrak{a}	+ \Vert x-y \Vert^2_\mathfrak{a},
\end{align}
we get  
\begin{align*}
	2\, &\mathfrak{a}(z^{n+2}, 3z^{n+2} - 4z^{n+1} + z^n)\\
	&= 2\, \mathfrak{a}(z^{n+2}, 3 Dz^{n+2} - Dz^{n+1})\\
	&= 4\, \mathfrak{a}(z^{n+2}, Dz^{n+2}) 
	+ 2\, \mathfrak{a}(z^{n+2}, Dz^{n+2} - Dz^{n+1})\\
	&= 4\, \mathfrak{a}(z^{n+2}, Dz^{n+2}) 
	+ 2\, \mathfrak{a}(Dz^{n+2}, Dz^{n+2} - Dz^{n+1}) \\
	&\qquad\qquad- 2\, \mathfrak{a}(z^{n+1}, Dz^{n+1})
	- 2\, \mathfrak{a}(z^{n+1}, z^{n+1}-z^{n+2})\\
	&= 3\,\Vert z^{n+2} \Vert^2_\mathfrak{a} - 4\,\Vert z^{n+1} \Vert^2_\mathfrak{a}
	+ \Vert z^{n} \Vert^2_\mathfrak{a}
	+ 2\,\Vert Dz^{n+2} \Vert^2_\mathfrak{a}
	- 2\,\Vert Dz^{n+1} \Vert^2_\mathfrak{a}
	+ \Vert Dz^{n+2}-Dz^{n+1} \Vert^2_\mathfrak{a},  
\end{align*}
which completes the proof. 
\end{proof}
After this preparation, we are now able to formulate the main convergence theorem.
\begin{theorem}[Convergence of BDF-$2$ for the delay equation~\eqref{eq:parabolicDelay}]
\label{thm:BDF2}
Let $\calBB\colon \Q\to\Q^*$ and $\calCC\colon \cHQ\to\cHQdual$ be symmetric, elliptic, and continuous in the respective spaces. Moreover, let $\omega$ denote the continuity constant of $\calCC$ satisfying $\omega\le 1/5$. 
Then, the BDF-$2$ scheme applied to~\eqref{eq:parabolicDelay}, i.e., the scheme
\[
  \bdf\, z^{n+2} 
  + 2\tau\, \calBB z^{n+2}
  + 2\,\calCC\, \bdf\, z^{n+1} - \calCC\, \bdf\, z^{n} 
  = 2\tau\, \tilde\h^{n+2}
\]
yields an approximation of second order, provided that $\tilde\h^{n+2}$ is a second-order approximation of $r(t^{n+2})$. To be precise, assuming a sufficiently smooth right-hand side~$\h$, a step size $\tau\le 1$, and initial data $z^0 = z(0),\,z^{-1} = \Phi(-\tau),\, z^{-2} = \Phi(-2\tau)$, we get 
\[
  \| z(t^n) - z^n \|^2_{\cHQ} + \tau \sum_{j=2}^n \| z(t^j) - z^j \|^2_\Q 
  \lesssim\  
  t^n \tau^4 + + t^nE_\mathrm{rhs} + E_\mathrm{init} 
\]
for $n \geq 2$, where $E_\mathrm{init} \coloneqq \|z(\tau)-z^{1}\|^2_{\cHQ} + \tau\, \| z(\tau)-z^{1}\|_{\Q}^2$ contains the initial error and $E_\mathrm{rhs} := \max_{j=2,\ldots,n} \|r(t^j) - \tilde\h^j\|^2_{\cHQ}$. 
\end{theorem}
\begin{proof}	
Inserting the exact solution of~\eqref{eq:parabolicDelay} within the numerical scheme, we obtain the defect equation
\begin{align*}
	\tfrac{1}{2\tau} \bdf z(t^{n+2}) &
	+ \calBB z(t^{n+2})
	+ \tfrac{1}{\tau} \calCC\, \big( \bdf z(t^{n+1}) \big)
	- \tfrac{1}{2\tau} \calCC\, \big( \bdf z(t^{n}) \big) \\
	&\quad = \tfrac{1}{2\tau} \bdf z(t^{n+2}) - \dot z(t^{n+2})
	+ \tfrac{1}{\tau} \calCC\, \big( \bdf z(t^{n+1}) \big)  \\
	&\qquad\quad- 2\, \calCC \dot z(t^{n+1}) - \tfrac{1}{2\tau} \calCC\, \big( \bdf z(t^{n}) \big) + \calCC \dot z(t^{n}) + r(t^{n+2})\\
	&\quad \eqqcolon d^{n+2} + r(t^{n+2})
\end{align*}	
with $d^{n+2} = \calO(\tau^2)$ by \Cref{lem:BDF}. 
With $e^n \coloneqq z(t^n)-z^n$, we get 
\begin{equation}\label{eq:BDFerrorRep1}
  %
  \bdf e^{n+2} 
  + 2\tau\, \calBB e^{n+2}
  + 2\, \calCC\, \bdf e^{n+1} 
  - \calCC\, \bdf e^{n} 
  = 2\tau \tilde d^{n+2}
\end{equation}
with $\tilde d^{n+2} = d^{n+2} + r(t^{n+2}) - \tilde\h^{n+2}$. 
Note that, due to the assumptions on the history function and the initial data, we have $e^{-2}=e^{-1}=e^0=0$. 
In the following, we write $\|\bullet \|_{\tilde b}$ for the norm induced by the operator~$\calBB$, which is equivalent to the $\Q$-norm, and $\|\bullet \|_{\tilde c}$ for the norm induced by $\calCC$. Note that the latter is equivalent to the $\cHQ$-norm with $\|\bullet\|^2_{\tilde c} \leq \omega\,\|\bullet\|^2$, where we use the short notation~$\|\bullet\| \coloneqq \|\bullet\|_\cHQ$. \medskip

{\bf Step 1:} 
In the first step, we derive an auxiliary estimate for differences of $e^n$. If we multiply~\eqref{eq:BDFerrorRep1} by 2 and apply $De^{n+2}$, we get 
\begin{equation}\label{eq:BDFerrorRep2}
	\begin{aligned}
	&\underbrace{2\,\langle \bdf e^{n+2}, D e^{n+2}\rangle}_{\eqqcolon\,T_1} 
	+ \underbrace{4\tau\,\langle \calBB e^{n+2}, D e^{n+2}\rangle}_{\eqqcolon\,T_2}
	\\&\qquad\qquad+ \underbrace{4\,\langle \calCC\, \bdf e^{n+1}, D e^{n+2}\rangle}_{\eqqcolon\,T_3} 
	- \underbrace{2\, \langle \calCC\, \bdf e^{n}, D e^{n+2}\rangle}_{\eqqcolon\,T_4}
	= 4\tau\, \langle \tilde d^{n+2}, D e^{n+2}\rangle.
	\end{aligned}
\end{equation}
Reformulating the terms $T_1$ and $T_2$ using~\eqref{eq:symBilinearForm} yields
\begin{equation*} 
  T_1
  = 4\, \|D e^{n+2}\|^2 + 2\, \langle D^2 e^{n+2}, D e^{n+2}\rangle
  = 5\, \|D e^{n+2}\|^2 - \|D e^{n+1}\|^2 + \|D^2 e^{n+2}\|^2 
\end{equation*}
and 
\begin{equation*} 
  T_2
  = 2\tau\, \big( \| e^{n+2}\|_{\tilde b}^2 - \| e^{n+1}\|_{\tilde b}^2 + \|D e^{n+2}\|_{\tilde b}^2 \big).
\end{equation*}
With $D e^{n+2} = D e^{n+1} + D^2 e^{n+2}$, we have
\begin{align*}
  T_3
  &= 8\, \langle \calCC\, D e^{n+1}, D e^{n+2}\rangle + 4\, \langle \calCC\, D^2 e^{n+1}, D e^{n+2}\rangle \\
  &= 8\, \langle \calCC\, D e^{n+1}, D e^{n+1} + D^2 e^{n+2}\rangle + 4\, \langle \calCC\, D^2 e^{n+1}, D e^{n+1} + D^2 e^{n+2}\rangle \\
  &= 8\, \| D e^{n+1}\|_{\tilde c}^2 + 2\, \|D e^{n+1}\|_{\tilde c}^2 - 2\, \|D e^{n}\|_{\tilde c}^2 + 2\, \|D^2 e^{n+1}\|_{\tilde c}^2 \\
  &\qquad + 8\, \langle \calCC\, D e^{n+1}, D^2 e^{n+2}\rangle + 4\, \langle \calCC\, D^2 e^{n+1}, D^2 e^{n+2}\rangle 
\end{align*}
and
\begin{align*}
  T_4
  \le 4\, \|D e^{n}\|_{\tilde c} \|D e^{n+2}\|_{\tilde c} + 2\, \|D^2 e^{n}\|_{\tilde c} \|D e^{n+2}\|_{\tilde c}. 
\end{align*}
The above computations inserted in~\eqref{eq:BDFerrorRep2} yield
\begin{align*}
  4\, \|D &e^{n+2}\|^2 + \|D^2 e^{n+2}\|^2 
  + 2\tau\, \|D e^{n+2}\|_{\tilde b}^2 
  + 8\, \| D e^{n+1}\|_{\tilde c}^2 + 2\, \|D^2 e^{n+1}\|_{\tilde c}^2 \\
  &\qquad+ \|D e^{n+2}\|^2 - \|D e^{n+1}\|^2
  + 2\tau\, \| e^{n+2}\|_{\tilde b}^2 - 2\tau\, \| e^{n+1}\|_{\tilde b}^2
  + 2\, \|D e^{n+1}\|_{\tilde c}^2 - 2\, \|D e^{n}\|_{\tilde c}^2 \\
  %
  &\le 4\tau\, \langle \tilde d^{n+2}, D e^{n+2} \rangle
  + 8\, \| D e^{n+1}\|_{\tilde c} \| D^2 e^{n+2}\|_{\tilde c} 
  + 4\, \| D^2 e^{n+1}\|_{\tilde c} \| D^2 e^{n+2}\|_{\tilde c} \\
  &\qquad+ 4\, \|D e^{n}\|_{\tilde c} \|D e^{n+2}\|_{\tilde c} + 2\, \|D^2 e^{n}\|_{\tilde c} \|D e^{n+2}\|_{\tilde c} \\
  &\le 4\tau\, \| \tilde d^{n+2}\|^2 + \tau\, \| D e^{n+2}\|^2 
  + 4\delta\, \| D e^{n+1}\|_{\tilde c}^2 + \tfrac4\delta\omega\, \| D^2 e^{n+2}\|^2
  + 2\gamma\, \| D^2 e^{n+1}\|^2_{\tilde c}  \\
  &\qquad+ \tfrac2\gamma\omega\, \| D^2 e^{n+2}\|^2+ \tfrac2\alpha\, \|D e^{n}\|_{\tilde c}^2 + 2\alpha\omega\, \|D e^{n+2}\|^2 + \tfrac1\beta\|D^2 e^{n}\|^2_{\tilde c} + \beta\omega\,\|D e^{n+2}\|^2,  
\end{align*}
where we use the weighted Young inequality four times with positive constants $\alpha,\beta,\gamma,\delta$. 
Rearranging terms leads to
\begin{equation}\label{eq:coeffs}
\begin{aligned}
(4 &- \tau - 2\alpha\omega - \beta\omega)\, \|D e^{n+2}\|^2 + (1-\tfrac2\gamma\omega - \tfrac4\delta\omega)\,\|D^2 e^{n+2}\|^2 
+ 2\tau\, \|D e^{n+2}\|_{\tilde b}^2 
\\&\quad + (10-4\delta)\, \| D e^{n+1}\|_{\tilde c}^2 - (2+\tfrac2\alpha)\, \|D e^{n}\|_{\tilde c}^2  + (2-2\gamma)\, \|D^2 e^{n+1}\|_{\tilde c}^2 - \tfrac1\beta\, \|D^2 e^{n+1}\|_{\tilde c}^2\\
&\quad+ \|D e^{n+2}\|^2 - \|D e^{n+1}\|^2
+ 2\tau\, \| e^{n+2}\|_{\tilde b}^2 - 2\tau\, \| e^{n+1}\|_{\tilde b}^2
\\
&\le 4\tau\, \| \tilde d^{n+2}\|^2. 
\end{aligned}
\end{equation}
We now set $\alpha = 7/8$, $\beta = 11/2$, $\gamma = 10/11$, and $\delta = 10/7$. This leads to 
\begin{align*}
	(4 - \tau - \tfrac{29}{4}\omega)\, \|D &e^{n+2}\|^2 + (1-5\omega)\,\|D^2  e^{n+2}\|^2 
	+ 2\tau\, \|D e^{n+2}\|_{\tilde b}^2 
	\\&\qquad + \tfrac{30}{7}\, \| D e^{n+1}\|_{\tilde c}^2 - \tfrac{30}{7}\, \|D e^{n}\|_{\tilde c}^2  + \tfrac{2}{11}\, \|D^2 e^{n+1}\|_{\tilde c}^2 - \tfrac{2}{11}\, \|D^2 e^{n+1}\|_{\tilde c}^2\\
	&\qquad+ \|D e^{n+2}\|^2 - \|D e^{n+1}\|^2
	+ 2\tau\, \| e^{n+2}\|_{\tilde b}^2 - 2\tau\, \| e^{n+1}\|_{\tilde b}^2
	\\
	&\le 4\tau\, \| \tilde d^{n+2}\|^2.  
\end{align*}
Assuming $\omega\le 1/5$ and $\tau \leq 1$, we therefore get 
\begin{equation*}
\begin{aligned}
  \|D e^{n+2}\|^2 
  + &2\tau\, \|D e^{n+2}\|_{\tilde b}^2 
  + \|D e^{n+2}\|^2 - \|D e^{n+1}\|^2
	+ 2\tau\, \| e^{n+2}\|_{\tilde b}^2 - 2\tau\, \| e^{n+1}\|_{\tilde b}^2\\
	&+ \tfrac{30}{7}\, \|D e^{n+1}\|_{\tilde c}^2 - \tfrac{30}{7}\, \|D e^{n}\|_{\tilde c}^2
    + \tfrac{2}{11}\,\|D^2 e^{n+1}\|_{\tilde c}^2 - \tfrac{2}{11}\|D^2 e^{n}\|^2_{\tilde c} \\&\hspace{5cm}
  \le C\, \tau^5 + C\,\tau\,\|r(t^{n+2}) - \tilde\h^{n+2}\|^2.
\end{aligned}
\end{equation*}
Building the sum over $n$, we get with $e^{-2}=e^{-1}=e^0=0$ and the definition $E_\mathrm{rhs} = \max_{j = 2,\ldots,n} \|r(t^j) - \tilde \h^j\|^2$ that
\begin{equation*}
\begin{aligned}
\sum_{j=2}^n \|D &e^{j}\|^2 + 2\tau \sum_{j=2}^n\|D e^{j}\|_{\tilde b}^2 + \|D e^{n}\|^2 + 2\tau\, \| e^{n}\|_{\tilde b}^2 + \tfrac{30}{7}\,\|D e^{n-1}\|^2_{\tilde c} + \tfrac{2}{11}\,\|D^2 e^{n-1}\|^2_{\tilde c} \\
& \leq  C t^n \tau^4 + Ct^nE_\mathrm{rhs} + \|D e^{1}\|^2 + 2\tau\, \| e^{1}\|_{\tilde b}^2 + \tfrac{30}{7}\,\|D e^{0}\|^2_{\tilde c} + \tfrac{2}{11}\,\|D^2 e^{0}\|^2_{\tilde c} \\
& \leq C t^n \tau^4 + Ct^nE_\mathrm{rhs} 
+ 2\, \|e^{1}\|^2 + 2\tau\, \| e^{1}\|_{\tilde b}^2.
\end{aligned}
\end{equation*}
In particular, we obtain with $E_\mathrm{init}$ introduced in the statement of the theorem that $\sum_{j=2}^n \|D e^{j}\|^2 \le C\, (t^n \tau^4 + t^nE_\mathrm{rhs} + E_\mathrm{init})$. \medskip

{\bf Step 2:} For the desired estimate of the error itself, we go back to~\eqref{eq:BDFerrorRep1}, multiply the equation by 2, and apply $e^{n+2}$. This leads to  
\begin{equation}\label{eq:BDFerrorRep5}
\begin{aligned}
	&\underbrace{2\,\langle \bdf e^{n+2}, e^{n+2}\rangle}_{\eqqcolon\,T_1} 
	+ \underbrace{4\tau\,\langle \calBB e^{n+2}, e^{n+2}\rangle}_{\eqqcolon\,T_2}
	\\&\qquad\qquad+ \underbrace{4\,\langle \calCC\, \bdf e^{n+1}, e^{n+2}\rangle}_{\eqqcolon\,T_3} 
	-\underbrace{2\, \langle\calCC\, \bdf e^{n}, e^{n+2}\rangle}_{\eqqcolon\,T_4}
	= 4\tau\, \langle \tilde d^{n+2}, e^{n+2}\rangle.
\end{aligned}
\end{equation}
With \Cref{lem:symBilinear:k2}, we can rewrite $T_1$ as 
\begin{equation*}
  T_1
  = \bdf \|e^{n+2}\|^2 + 2\, \|D e^{n+2}\|^2 - 2\, \|D e^{n+1}\|^2 + \| D^2 e^{n+2} \|^2. 
\end{equation*}
For the second term, we directly get $T_2 = 4\tau\, \|e^{n+2}\|_{\tilde b}^2$. 
The third term is simplified using $e^{n+2} = e^{n+1} + D e^{n+2}$ and once more \Cref{lem:symBilinear:k2}, leading to 
\begin{align*}
  T_3
  &= 4\, \langle \calCC\, \bdf e^{n+1}, e^{n+1} \rangle
  + 4\, \langle \calCC\, (2D+D^2) e^{n+1}, De^{n+2} \rangle \\
  &= 2\, \bdf \|e^{n+1}\|^2_{\tilde c} + 4\, \|D e^{n+1}\|^2_{\tilde c} - 4\, \|D e^{n}\|^2_{\tilde c} + 2\, \| D^2 e^{n+1} \|^2_{\tilde c} \\
  &\hspace{4cm} + 8\, \langle \calCC D e^{n+1}, De^{n+2} \rangle + 4\, \langle \calCC D^2 e^{n+1}, De^{n+2} \rangle.
\end{align*}
Finally, using $e^{n+2} = D e^{n+2} + D e^{n+1} + e^{n}$ and \Cref{lem:symBilinear:k2}, the last term can be written as 
\begin{align*}
T_4
= \bdf \Vert e^{n} \Vert^2_{\tilde c} + 2\, \Vert D e^{n} \Vert^2_{\tilde c} - 2\, \Vert D e^{n-1} \Vert^2_{\tilde c} + \Vert D^2 e^{n} \Vert^2_{\tilde c}+  2\, \langle \calCC\, \bdf e^{n}, De^{n+2} + De^{n+1}\rangle.
\end{align*}
Using the above expressions, equation~\eqref{eq:BDFerrorRep5} yields 
\begin{align*}
  \bdf \Vert e^{n+2} \Vert^2 &
  + 2\, \Vert D e^{n+2} \Vert^2 - 2\, \Vert D e^{n+1} \Vert^2
  + \Vert D^2 e^{n+2} \Vert^2 
  + 4\tau\, \|e^{n+2}\|_{\tilde b}^2 \\
  &\qquad+ 2\,\bdf \Vert e^{n+1} \Vert^2_{\tilde c} 
  + 4\, \Vert D e^{n+1} \Vert^2_{\tilde c} - 4\, \Vert D e^{n} \Vert^2_{\tilde c} + 2\, \Vert D^2 e^{n+1} \Vert^2_{\tilde c} \\
  &\qquad- \bdf \Vert e^{n} \Vert^2_{\tilde c} - 2\, \Vert D e^{n} \Vert^2_{\tilde c} + 2\, \Vert D e^{n-1} \Vert^2_{\tilde c} - \Vert D^2 e^{n} \Vert^2_{\tilde c} \\
  %
  &\le 4\tau\, \|\tilde d^{n+2}\|\, \|e^{n+2}\| 
  + 8\, \|D e^{n+1}\|_{\tilde c} \|De^{n+2}\|_{\tilde c} + 4\, \|D^2 e^{n+1}\|_{\tilde c} \|De^{n+2}\|_{\tilde c} \\
  &\qquad+ 2\, \|2D e^{n} + D^2e^n\|_{\tilde c} \| D e^{n+2} + D e^{n+1}\|_{\tilde c} \\
  &\le C\tau\, \|\tilde d^{n+2}\|^2 + 2\tau\, \|e^{n+2}\|_{\tilde b}^2 
  + 4\, \|D e^{n+1}\|^2_{\tilde c} + 2\, \|D^2 e^{n+1}\|^2_{\tilde c} + 6\, \|De^{n+2}\|^2_{\tilde c} \\
  &\qquad+ 8\, \|D e^{n}\|^2_{\tilde c} + 2\, \| D^2e^n\|^2_{\tilde c} + 2\, \| D e^{n+2}\|^2_{\tilde c} + 2\, \|D e^{n+1}\|^2_{\tilde c}
\end{align*}
for some constant $C$ that depends on the ellipticity constant of~$\calBB$. With 
\begin{equation*}
2\,\bdf \Vert e^{n+1} \Vert^2_{\tilde c} - \bdf \Vert e^{n} \Vert^2_{\tilde c} = 6\,\Vert e^{n+1} \Vert^2_{\tilde c} -11\, \Vert e^{n} \Vert^2_{\tilde c} + 6\,\Vert e^{n-1} \Vert^2_{\tilde c} - \Vert e^{n-2} \Vert^2_{\tilde c},
\end{equation*}
we get 
%
\begin{align*}
  \bdf& \Vert e^{n+2} \Vert^2 + \Vert D^2 e^{n+2} \Vert^2 + \Vert D^2 e^{n+1} \Vert^2_{\tilde c} + 2\tau\, \|e^{n+2}\|_{\tilde b}^2 \\
  &\qquad + 2\, \Vert D e^{n+2} \Vert^2 - 2\, \Vert D e^{n+1} \Vert^2
  + \Vert D^2 e^{n+1} \Vert^2_{\tilde c} - \Vert D^2 e^{n} \Vert^2_{\tilde c} + 4\, \Vert D e^{n+1} \Vert^2_{\tilde c}  \\
  &\qquad - 6\, \Vert D e^{n} \Vert^2_{\tilde c} + 2\, \Vert D e^{n-1} \Vert^2_{\tilde c} 
  + 6\, \Vert e^{n+1} \Vert^2_{\tilde c} - 11\, \Vert e^{n} \Vert^2_{\tilde c} 
  + 6\, \Vert e^{n-1} \Vert^2_{\tilde c} - \Vert e^{n-2} \Vert^2_{\tilde c} \\
  &\le C\tau\, \|\tilde d^{n+2}\|^2 
  + 8\, \|D e^{n}\|^2_{\tilde c} + 6\, \|D e^{n+1}\|^2_{\tilde c} + 8\, \|D e^{n+2}\|^2_{\tilde c}  \\ 
  &\qquad+ 2\, \| D^2e^n\|^2_{\tilde c} + 2\, \|D^2 e^{n+1}\|^2_{\tilde c} \\
  &\le C\tau\, \|\tilde d^{n+2}\|^2 
  + 4\, \|D e^{n-1}\|^2_{\tilde c} + 16\, \|D e^{n}\|^2_{\tilde c} + 10\, \|D e^{n+1}\|^2_{\tilde c} + 8\, \|D e^{n+2}\|^2_{\tilde c}. 
\end{align*}
%
Dropping the terms $\|D^2e^{n+2}\|^2$ and $\|D^2 e^{n+1}\|_{\tilde c}^2$ on the left-hand side, summing up, and using $e^{-2}=e^{-1}=e^0=0$ and $\omega \leq 1/5$, we obtain
%
\begin{align*}
3\,\|e^{n}&\|^2 - \|e^{n-1}\|^2 + 2\,\|De^{n}\|^2 + \|D^2e^{n-1}\|_{\tilde c}^2 + 4\,\|De^{n-1}\|_{\tilde c}^2 - 2\,\|De^{n-2}\|_{\tilde c}^2 \\&\qquad+ 6\,\|e^{n-1}\|_{\tilde c}^2 - 5\, \|e^{n-2}\|_{\tilde c}^2 + \|e^{n-3}\|_{\tilde c}^2 + 2\tau \sum_{j=2}^n\|e^j\|^2_{\tilde b}\\
& \leq C \tau \sum_{j=2}^n\|\tilde d^{j}\|^2 + \sum_{j=2}^n\Big(4\,\|De^{j-3}\|_{\tilde c}^2 + 16\,\|De^{j-2}\|_{\tilde c}^2 + 10\,\|De^{j-1}\|_{\tilde c}^2 +8\,\|De^{j}\|_{\tilde c}^2\Big) \\
& \qquad + 3\,\|e^1\|^2 - \|e^0\|^2 + 2\,\|De^1\|^2 + \|D^2e^0\|_{\tilde c}^2 + 4\,\|De^0\|_{\tilde c}^2 - 2\,\|De^{-1}\|_{\tilde c}^2 
\\&\qquad+ 6\,\|e^0\|^2 - 5\,\|e^{-1}\|^2 + \|e^{-2}\|^2 \\ 
%
%
& \le C \tau \sum_{j=2}^n\|\tilde d^{j}\|^2 + 38\omega\, \sum_{j=2}^n \|De^{j}\|^2 
+ 3\,\|e^1\|^2 
+ 8\,\|De^1\|^2 \\
%
& \le C \tau \sum_{j=2}^n\|\tilde d^{j}\|^2 + 8 \sum_{j=2}^n \|De^{j}\|^2 
+ 11\,\|e^1\|^2. 
\end{align*}
%
%
Recalling $\tilde d^j = \mathcal{O}(\tau^2) + E_\mathrm{rhs}$ 
and applying the estimate obtained in {\bf Step 1} of this proof, namely $\sum_{j=2}^n \|D e^{j}\|^2 \le C\, (t^n \tau^4 + t^n E_\mathrm{rhs} + E_\mathrm{init})$, we obtain 
\begin{equation*}
\begin{aligned}
3\,\|e^{n}&\|^2 - \|e^{n-1}\|^2 + 2\,\|De^{n}\|^2 + \|D^2e^{n-1}\|_{\tilde c}^2 + 4\,\|De^{n-1}\|_{\tilde c}^2 - 2\,\|De^{n-2}\|_{\tilde c}^2 \\&\qquad+ 6\,\|e^{n-1}\|_{\tilde c}^2 - 5\, \|e^{n-2}\|_{\tilde c}^2 + \|e^{n-3}\|_{\tilde c}^2 + 2\tau \sum_{j=2}^n\|e^j\|^2_{\tilde b}\\
& \leq \tilde C\, ( t^n \tau^4 + t^n E_\mathrm{rhs} + E_\mathrm{init}). 
\end{aligned}
\end{equation*}
Dropping the terms $\|De^{n}\|^2$, $\|D^2e^{n-1}\|_{\tilde c}^2$, and $\|e^{n-3}\|_{\tilde c}^2$ on the left-hand side, we obtain
\begin{equation}\label{eq:BDFerrorRep9}
\begin{aligned}
3\,\|e^{n}&\|^2  + 4\,\|De^{n-1}\|_{\tilde c}^2 + 6\,\|e^{n-1}\|_{\tilde c}^2 + 2\tau \sum_{j=2}^n\|e^j\|^2_{\tilde b}\\
&\leq \|e^{n-1}\|^2 + 2\,\|De^{n-2}\|_{\tilde c}^2 + 5\, \|e^{n-2}\|_{\tilde c}^2 
+ \tilde C\, ( t^n \tau^4 + E_\mathrm{init})  \\
&\leq \frac56\, \Big( 3\,\|e^{n-1}\|^2  + 4\,\|De^{n-2}\|^2_{\tilde c} + 6\,\|e^{n-2}\|_{\tilde c}^2+ 2\tau \sum_{j=2}^{n-1}\|e^j\|^2_{\tilde b} \Big) \\&\hspace{6.5cm}
+ \tilde C\, ( t^n \tau^4 + t^n E_\mathrm{rhs} + E_\mathrm{init})
\end{aligned}
\end{equation}
for all $n\ge 2$. Using the estimate in~\eqref{eq:BDFerrorRep9} multiple times, we get with $\sum_{j=0}^\infty \big(\frac56\big)^j = 6$ that
\begin{align*}
3\,\|e^{n}\|^2  + 4\,\|De^{n-1} \|_{\tilde c}^2 + 6\,\|e^{n-1}\|_{\tilde c}^2 + 2\tau \sum_{j=2}^n\|e^j\|^2_{\tilde b}
\leq 3\,\|e^{1}\|^2  
+ 6\,\tilde C\, ( t^n \tau^4 + t^n E_\mathrm{rhs} + E_\mathrm{init}). 
\end{align*}
Since the first term on the right-hand side can be once again bounded in terms of $E_\mathrm{init}$, this is the assertion.
\end{proof}
\begin{remark}\label{rem:relaxomega}
The choice of the parameters~$\alpha$, $\beta$, $\gamma$, and $\delta$ in~\eqref{eq:coeffs} can be further improved, leading to a relaxed condition on $\omega$. To balance the respective terms, we require that $5-2\delta=1+1/\alpha$, $2-2\gamma=1/\beta$, as well as
\[
	4 - \tau - 2\alpha\omega - \beta \omega > 0, \qquad
	1 - 2/\gamma\omega - 4/\delta\omega \geq 0
\]
for reasonably small values of $\tau$. This restricts possible choices, such that the condition on $\omega$ can only be slightly improved. Nearly optimal values can be obtained by the solution of a constrained optimization problem. As an example, under the (more restrictive) assumption that $\tau \le 1/4$, the choice $\alpha = 15/4$ and $\beta = 15/2$ (and thus $\gamma = 14/15$, $\delta = 28/15$) leads to the improved condition $\omega \le 7/30$. 
\end{remark}
%
%
\section{Numerical Experiments}\label{sec:num}
This section is devoted to the numerical illustration of the convergence result presented in~\Cref{thm:convergence} and the necessity of a weak coupling condition. Moreover, we present a semi-explicit method of order three based on the above construction. 
%
%
\subsection{Poroelastic example}\label{sec:num:poroEx}
In the first experiment, we investigate the convergence rates of the semi-explicit second-order scheme~\eqref{eqn:semiExplicit:order2} and compare the results with an implicit second-order scheme based on a BDF-$2$ discretization. We choose $\Omega = (0,1)^2$, $T=1$, and consider the poroelastic parameters of {\em Charcoal granite} in combination with water (see \Cref{tab:examplesWeakCoupling} or~\cite[Tab.~4]{DetC93}), i.e., we set
\begin{equation*}
\lambda = 2.23\cdot 10^{10}, \quad 
\mu = 1.9\cdot 10^{10}, \quad 
\alpha = 0.27, \quad 
M = 8.5\cdot10^{10},\quad 
\kappa/\nu = 1.0\cdot10^{-19}.
\end{equation*}
Further, the right-hand sides are given by   
\begin{equation*}
f \equiv [\ 1\ \ 2\ ]^T, \qquad 
g(t,x) = 30\,\sin(2\pi\, t\,x_1 + 4\pi\,t) 
\end{equation*}
and the initial condition reads $p^0(x) = 50\,x_1(1-x_1)x_2(1-x_2)$. Accordingly, $u^0$ is defined through the consistency condition~\eqref{eq:poroAbstract:a}, and $p^1, u^1$ by an implicit Euler step as described in \Cref{rem:initDataImplEuler}. 
Note that with the above parameters, it holds that 
\begin{equation*} 
	\omega 
	= \alpha^2M/(\mu+\lambda) 
	\approx 0.15 
	< 1/5
\end{equation*} 
such that the coupling condition in \Cref{thm:convergence} is just fulfilled.  

The computations are based on a finite element implementation in FEniCS, leading to a system as described shortly in \Cref{sec:model:space}. 
We now investigate the convergence behavior of the semi-explicit scheme~\eqref{eqn:semiExplicit:order2} and compare it with a second-order implicit BDF discretization. For the computation of a reference solution, we choose an implicit midpoint scheme with step size $\tau_\mathrm{ref} = 2^{-11}$ and a spatial mesh width~$h_\mathrm{ref} = 2^{-7}$. Since we are mainly interested in the temporal discretization errors, we compute the second-order schemes for step sizes $\tau \in \{2^{-2},\dots,2^{-9}\}$ with the fixed spatial parameter~$h=2^{-7}$. 

The results are presented in \Cref{fig:physicalEx}. Therein, we use the notion $p(T)$ for the reference solution and $p^N_h$ for the discrete solution at time $T=N\tau$ (and accordingly for~$u$). We observe second-order convergence for both the implicit and the semi-explicit scheme. The implicit method, however, achieves slightly better results compared to the semi-explicit one. For comparison, we also included the semi-explicit scheme of first order; see~\eqref{eqn:semiExpl:Ord1}. The main advantage of the semi-explicit scheme lies in the fact that the two poroelastic equations can be solved sequentially, which results in a computational speedup. Moreover, standard preconditioners for elliptic and parabolic systems can be used. Note, however, that the semi-explicit method is only stable if an appropriate coupling condition is fulfilled as indicated in \Cref{thm:convergence}. This is further investigated in the following subsection.

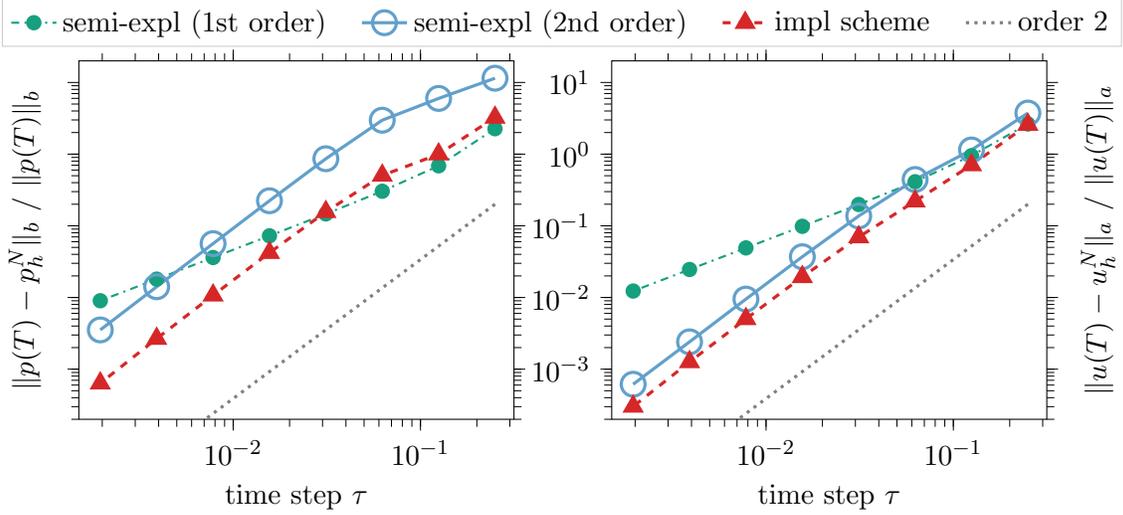
\begin{figure}
	\centering
	\ref{legPoroConvergence}\\ 
\begin{tikzpicture}

\begin{axis}[
width=2.9in,
height=2.5in,
log basis x={10},
log basis y={10},
tick align=outside,
tick pos=both,
x grid style={white!69.01960784313725!black},
xlabel={time step $\tau$},
xmin=0.0015, xmax=0.315572206672458,
xmode=log,
xtick style={color=black},
y grid style={white!69.01960784313725!black},
ylabel={$\|p(T)-p_h^N\|_{b}\  /\ \|p(T)\|_{b}$},
yticklabels = {},
y label style={at={(-.06,0.5)}},
ymin=0.0002, ymax=20,
ymode=log,
ytick style={color=black},
legend cell align={left},
legend columns = 4,
legend style={at={(0.30,1.15)}, anchor=north west, draw=white!80.0!black},
legend to name=legPoroConvergence,
]
\addplot [thick, color0!70!green, mark=*, mark size=2.5, mark options={solid}, dashdotted]
table {%
0.25 2.2622186326998
0.125 0.687117759772629
0.0625 0.303704566966348
0.03125 0.146598164391984
0.015625 0.0725130806070638
0.0078125 0.0361252101591253
0.00390625 0.0180385381887363
0.001953125 0.00901503836499394
};
\addlegendentry{semi-expl (1st order)$\quad$}

\addplot [very thick, color0!70!white, mark=o, mark size=4.5, mark options={solid}]
table {%
0.25 11.4528994939681
0.125 5.98644722401631
0.0625 2.97409972229476
0.03125 0.862683605508286
0.015625 0.223706061851106
0.0078125 0.0564364236965791
0.00390625 0.0141361284869837
0.001953125 0.00353079304104601
};
\addlegendentry{semi-expl (2nd order)$\quad$}

\addplot [very thick, color3, dashed, mark=triangle*, mark size=3.5, mark options={solid}]
table {%
0.25 3.2413210948058
0.125 0.997045457702548
0.0625 0.504237900518704
0.03125 0.156622618940424
0.015625 0.0420791432547027
0.0078125 0.010726360941748
0.00390625 0.00266824191188737
0.001953125 0.0006367208861766
};
\addlegendentry{impl scheme$\quad$}

\addplot [very thick, gray, dotted]
table {%
	0.25 0.2
	0.0019 0.00001552
};
\addlegendentry{order 2}

\end{axis}

\end{tikzpicture}\hspace{-.8em}
\begin{tikzpicture}

\begin{axis}[
width=2.9in,
height=2.5in,
log basis x={10},
log basis y={10},
tick align=outside,
tick pos=both,
x grid style={white!69.01960784313725!black},
xlabel={time step \(\displaystyle \tau\)},
xmin=0.0015, xmax=0.315572206672458,
xmode=log,
xtick style={color=black},
y grid style={white!69.01960784313725!black},
ylabel={$\|u(T)-u_h^N\|_{a}\  /\ \|u(T)\|_{a}$},
y label style={at={(1.18,0.5)}},
ymin=0.0002, ymax=20,
ymode=log,
ytick style={color=black}
]

\addplot [thick, color0!70!green, mark=*, mark size=2.5, mark options={solid}, dashdotted]
table {%
0.25 2.63447344682214
0.125 0.95174740426134
0.0625 0.411826891854543
0.03125 0.198291987483681
0.015625 0.0984189169942012
0.0078125 0.0491768969723484
0.00390625 0.0245980748549917
0.001953125 0.0123027081284418
};

\addplot [very thick, color0!70!white, mark=o, mark size=4.5, mark options={solid}]
table {%
0.25 3.74262067999222
0.125 1.14367190489737
0.0625 0.442328275064346
0.03125 0.137051935060564
0.015625 0.0370378663310186
0.0078125 0.00950290910828976
0.00390625 0.00240626825668776
0.001953125 0.000616245983368695
};

\addplot [very thick, color3, dashed, mark=triangle*, mark size=3.5, mark options={solid}]
table {%
0.25 2.58243009355785
0.125 0.70428654944105
0.0625 0.220435964429632
0.03125 0.0696939137205959
0.015625 0.0193825762054349
0.0078125 0.00502745980093292
0.00390625 0.0012606242063587
0.001953125 0.000301932744171904
};

\addplot [very thick, gray, dotted]
table {%
	0.25 0.2
	0.0019 0.00001552
};
\end{axis}

\end{tikzpicture}
	\caption{Relative errors in $p$ (left, measured in the $b$-norm) and $u$ (right, measured in the $a$-norm) for the poroelastic example in \Cref{sec:num:poroEx} at the final time $T$ for fixed $h = 2^{-8}$ and varying $\tau$.}
	\label{fig:physicalEx}
\end{figure}
%
%
\subsection{Sharpness of the weak coupling condition}\label{sec:num:toyEx}
We now present a numerical example to investigate the requirement of the weak coupling condition in \Cref{thm:convergence}. To this end, we consider the following toy problem of the form~\eqref{eq:poroAbstract} with $\V = \cHV = \R^3$, $\Q = \cHQ = \R^1$ and bilinear forms
\begin{align*}
	a(u,v) = v^TAu,\qquad 
	d(v,p) = \sqrt{\omega}\, p^TDv, \qquad 
	c(p,q) = q^TCp, \qquad 
	b(p,q) = q^TBp
\end{align*} 
with matrices 
\begin{align*}
	A \vcentcolon= 
	\frac{1}{2 - \sqrt{2}}
	\begin{bmatrix}
		\ 2 & -1 & \ 0\\
		-1 &\ 2 & -1\\
		\ 0 & -1 &\ 2
	\end{bmatrix},\qquad 
	D \vcentcolon= \begin{bmatrix} \frac23 & \frac13 & \frac23 \end{bmatrix}, \qquad 
	C \vcentcolon= 1,\qquad 
	B\vcentcolon= 1.
\end{align*}
The prefactor of $A$ is chosen in such a way that $c_a$, which equals the smallest eigenvalue of $A$, is exactly~$1$. Moreover, we have $c_c = 1$ and for the continuity constant of $d$ we get~$C_d = \sqrt{\omega}$. Therefore, we consider as coupling parameter~$\omega = C_d^2 / (c_ac_c)$.

We test our semi-explicit scheme~\eqref{eqn:semiExplicit:order2} with different step sizes~$\tau$ and different coupling coefficients $\omega$. The relative errors compared to a fine discretization with an implicit midpoint rule with step size $\tau_\mathrm{ref} = 2^{-14}$ are computed at the final time $T = 1/2$. For the forcing functions, we choose
\begin{displaymath}
	f \equiv [\ 1\ \ 1\ \ 1\ ]^T\qquad\text{and}\qquad 
	g(t) = \sin(t).
\end{displaymath}
The corresponding results are presented in \Cref{fig:toyEx2}. While the sufficient condition from \Cref{thm:convergence} reads $\omega \leq 1/5$ and the delay approach from \Cref{sec:construction:stability} demands~$\omega < 1/3$, we observe that the critical value for stability is roughly $0.38$ and therefore slightly relaxed compared to the theoretical considerations. The experiment shows that a coupling condition as in \Cref{thm:convergence} is indeed necessary and  -- up to a moderate scaling factor -- rather sharp.  
\begin{figure}[t]
	\centering
	\input{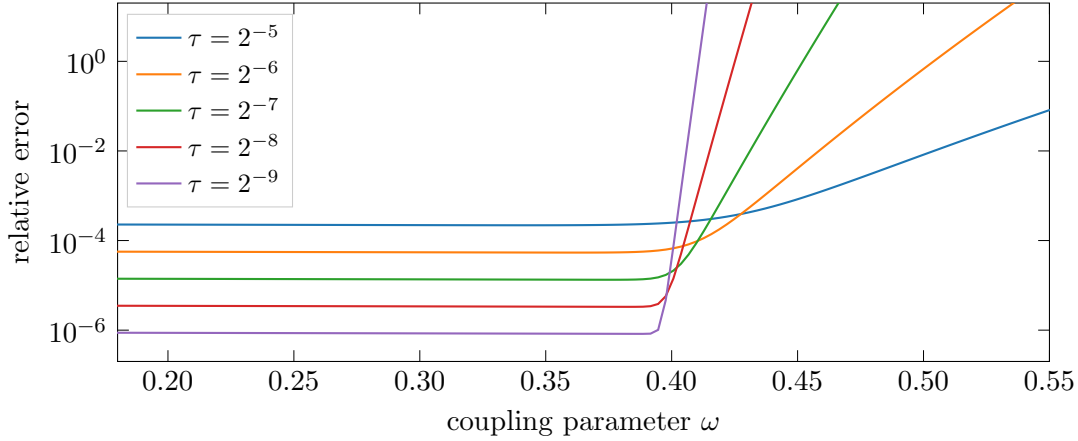}
	\caption{Relative errors of the second-order semi-explicit method at the final time point~$T=1/2$ for different coupling parameters $\omega$ and different time step sizes $\tau$.}
	\label{fig:toyEx2}
\end{figure}
%
%
\subsection{Semi-explicit scheme of order 3}\label{sec:num:order3}
As an outlook, we go beyond the presented theory and motivate a possible extension to a semi-explicit third-order scheme. This is done by using the BDF-$3$ scheme for the delay system in~\eqref{eqn:delay:Taylor}, see the discussion in \Cref{rem:higherOrderBDF}. 
For $k=3$, we have~$p \approx 3 p_\tau - 3 p_{2\tau} + p_{3\tau}$, which 
yields the semi-explicit 3-step scheme
\begin{align*}
\calA u^{n+3} - \calD^* \big(3p^{n+2}-3p^{n+1}+p^n\big) &= f^{n+3}, \\
\calD \tfrac{ 11u^{n+3} - 18u^{n+2} + 9u^{n+1} - 2u^n }{6\tau} + \calC \tfrac{ 11p^{n+3} - 18p^{n+2} + 9p^{n+1} - 2p^n }{6\tau} + \calB p^{n+3} &= g^{n+3}.
\end{align*}	
To illustrate the behavior in terms of the convergence rate and the weak coupling condition, we consider again the setting presented in~\Cref{sec:num:toyEx}. The corresponding results are shown in \Cref{fig:toyEx3}. Note that the error decreases roughly by a factor $8$ when halving the step size~$\tau$, which indicates a third-order convergence rate. As before, we observe that a suitably small coupling of the two equations is necessary in order to ensure stability. The numerically observed critical point for stability is roughly~$\omega \le 1/6$ and hence smaller than in the second-order case of \Cref{fig:toyEx2}. This indicates that the required coupling condition depends on the order~$k$ of the corresponding scheme. Performing a similar analysis of the corresponding delay equation as in \Cref{sec:construction:stability} yields that delay-independent asymptotic stability is (numerically) guaranteed for $\omega < 1/7$.
\begin{figure}[t]
	\centering
	\input{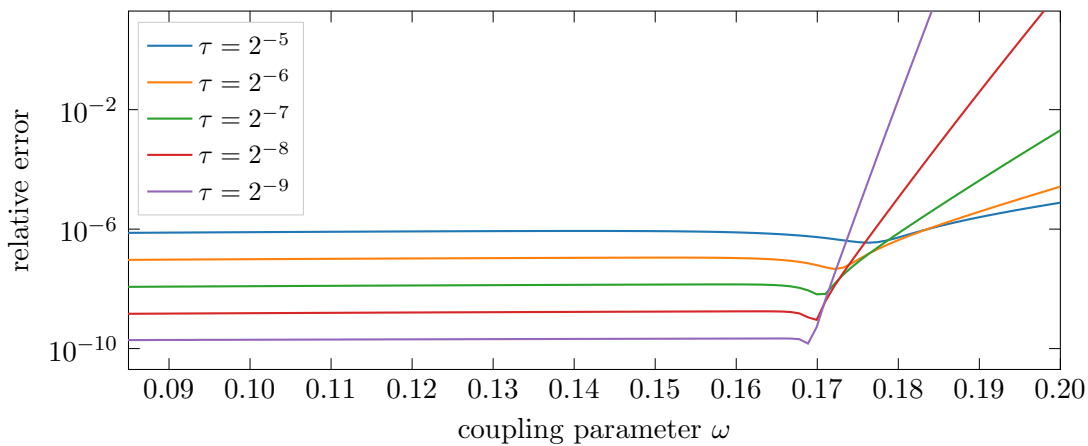}
	\caption{Relative errors of the third-order semi-explicit method at the final time point~$T=1/2$ for different coupling parameters $\omega$ and different time step sizes $\tau$.}
	\label{fig:toyEx3}
\end{figure}
%
%
\section{Conclusions}
Within this paper, we have constructed a semi-explicit second-order time-integration scheme for linear poroelasticity that decouples the problem and hence is suitable in a co-design paradigm where specialized legacy codes for the elliptic and parabolic equation can be used. The method is constructed by first perturbing the elastic equation with time delays, which equal multiples of the time step size, and then applying BDF-$2$ to this delay equation. We have proven convergence of this scheme under a suitable weak coupling condition. This coupling condition is, as in the first-order case~\cite{AltMU21}, explicitly quantified via an asymptotic stability analysis of the delay equation. While our work focuses on the second-order scheme, we have demonstrated in a numerical example that the same idea can also be used to construct a third-order scheme, which however requires a more restrictive weak coupling condition as well as an alternative convergence proof. 
%
%
\section*{Acknowledgments} 
This project is funded by the Deutsche Forschungsgemeinschaft (DFG, German Research Foundation) - 467107679. B.~Unger additionally acknowledges support by the Stuttgart Center for Simulation Science (SimTech). 
Moreover, major parts of this work were carried out while the first author was affiliated with the Institute of Mathematics and the Centre for Advanced Analytics and Predictive Sciences (CAAPS) at the University of Augsburg.
%
%
\bibliographystyle{alpha}
\bibliography{references}
%
%
\appendix 
\section{}\label{appendix}
\begin{proposition}
\label{prop:delayTaylor}
Assume sufficiently smooth right-hand sides~$f$ and~$g$ and a history function~$\bar{\Phi}$ satisfying~\eqref{eqn:historyTaylor} such that the solution $(\bar{u},\bar{p})$ of the delay system \eqref{eqn:delay:Taylor} satisfies~${\bar p}\in W^{k+1,\infty}(\cHQ\!)$. 
Then, the solutions to~\eqref{eq:poroAbstract} and~\eqref{eqn:delay:Taylor} are equal up to a term of order~$\tau^k$, i.e., for almost all~$t\in[0,T]$ we have 
\begin{equation*}
	\Vert \bar p(t) - p(t) \Vert^2_\Q + \Vert \bar u(t) - u(t)\Vert^2_\V 
	\lesssim t\, \tau^{2k}\, \Big[ \Vert \bar{\Phi}\Vert_{W^{k+1,\infty}(-\tau,0;\cHQ\!)}^2 + \Vert {\bar p}\Vert_{W^{k+1,\infty}(\cHQ\!)}^2 \Big]. 
\end{equation*}
\end{proposition}
\begin{proof}
We define~$e_p \coloneqq \bar p - p$ and~$e_u \coloneqq \bar u - u$, which satisfy~$e_p(0)=0$ and~$e_u(0)=0$ due to the particular choice of the history function, cf.~equation~\eqref{eqn:historyTaylor}. By a Taylor expansion, we know that 
\begin{align*}
	\bar{p}(t)
	&= \sum_{j=0}^{k-1} \tfrac{\tau^{j}}{j!}\, \bar{p}^{(j)}(t-\tau)
	+ \int\limits_{t-\tau}^t \bar{p}^{(k)}(\xi)\tfrac{(t-\xi)^{k-1}}{(k-1)!}\,\mathrm{d}\xi, \\
	\dot{\bar{p}}(t)
	&= \sum_{j=0}^{k-1} \tfrac{\tau^{j}}{j!}\, \bar{p}^{(j+1)}(t-\tau)
	+ \int\limits_{t-\tau}^t \bar{p}^{(k+1)}(\xi)\tfrac{(t-\xi)^{k-1}}{(k-1)!}\,\mathrm{d}\xi.
\end{align*}
With this, the errors satisfy the system 
\begin{subequations}
\label{eqn:lemDelayTwoField}
\begin{align}
	a(e_u,v) - d(v, e_p) 
	&= - \int_{t-\tau}^t \tfrac{(t-\xi)^{k-1}}{(k-1)!}\, d(v, \bar{p}^{(k)}(\xi))\,\mathrm{d}\xi, \label{eqn:lemDelayTwoField:a} \\
	d(\dot e_u, q) + c(\dot e_p,q) + b(e_p,q) 
	&= 0 \label{eqn:lemDelayTwoField:b} 
\end{align}
\end{subequations}
for all test functions~$v\in\V$, $q\in\Q$.  
Moreover, considering the derivatives of~\eqref{eq:poroAbstract:a} and~\eqref{eqn:delay:Taylor:a}, we obtain 
\begin{align}
\label{eqn:lemDelayTwoField:derivative}
	a(\dot e_u,v) - d(v, \dot e_p) 
	= - \int_{t-\tau}^t \tfrac{(t-\xi)^{k-1}}{(k-1)!}\, d(v, \bar{p}^{(k+1)}(\xi))\,\mathrm{d}\xi.
\end{align}
The sum of~\eqref{eqn:lemDelayTwoField:derivative} with test function~$v=\dot e_u$ and~\eqref{eqn:lemDelayTwoField:b} with test function~$q=\dot e_p$, bounding the integral, and an application of Young's inequality yield
\begin{align*}
	\Vert \dot e_u\Vert_a^2 + \Vert \dot e_p\Vert_c^2 + \tfrac 12 \tddt  \Vert e_p \Vert^2_{ b}
	&= - \int_{t-\tau}^t \tfrac{(t-\xi)^{k-1}}{(k-1)!}\, d(\dot e_u, \bar{p}^{(k+1)}(\xi))\,\mathrm{d}\xi \\ 
	&\le \tfrac{\tau^{k}}{(k-1)!}\, C_d\, \Vert \dot{e}_u\Vert_\V  \Vert \bar{p}^{(k+1)} \Vert_{L^\infty(t-\tau,t;\cHQ)} \\
	&\le \tfrac{1}{2}\, \Vert \dot e_u\Vert_a^2 + C\, \tau^{2k}\, \Vert \bar{p}^{(k+1)} \Vert_{L^\infty(t-\tau,t;\cHQ)}^2 
\end{align*}
with the constant~$C=\frac{C_d^2}{2 c_a (\corr{(k-1)}!)^2}$. Note that we use the notion~$\|\bullet \|_{a}$, $\|\bullet \|_{b}$, and~$\|\bullet \|_{c}$ for the norms induced by the bilinear forms $a$, $b$, and $c$, respectively. Hence, we can eliminate $\Vert \dot e_u\Vert_a^2$ on the right-hand side. By the ellipticity of the bilinear forms and an integration over~$[0,t]$ we conclude that 
\begin{align}
\label{eqn:inProofDelay}
	\int_0^t \Vert \dot e_u(s)\Vert_\V^2 \ds + \Vert e_p(t) \Vert_\Q^2
	\lesssim \tau^{2k}\,t\, \Vert \bar{p}^{(k+1)}\Vert_{L^\infty(-\tau,t;\cHQ\!)}^2.
\end{align}
Note that we use here the convention that~$\bar{p}$ equals the history function~$\bar\Phi$ on~$[-\tau, 0]$. 
In the same way, the sum of~\eqref{eqn:lemDelayTwoField:a} with test function~$v=\dot e_u$ and~\eqref{eqn:lemDelayTwoField:b} with test function~$q=e_p$ yields the estimate 
\begin{align*}
	\tfrac 12\tddt \Vert e_u\Vert_a^2 + \tfrac 12\tddt \Vert e_p\Vert_c^2 + \Vert e_p\Vert_{ b}^2
	&= - \int_{t-\tau}^t \tfrac{(t-\xi)^(k-1)}{(k-1)!}\, d(\dot e_u, \bar{p}^{(k)}(\xi))\,\mathrm{d}\xi \\
	&\lesssim \Vert \dot e_u \Vert_\V^2
	+ \tau^{2k}\, \Vert \bar{p}^{(k)}\Vert_{L^\infty(t-\tau,t;\cHQ)}^2. 
\end{align*}
Integration over~$[0,t]$ and the application of estimate~\eqref{eqn:inProofDelay} finally gives 
\begin{equation*}
	\Vert e_u (t)\Vert_\V^2
	\lesssim t\, \tau^{2k}\, \Vert \bar{p}^{(k+1)}\Vert_{L^\infty(-\tau,t;\cHQ\!)}^2, 
\end{equation*}
which completes the proof. 
\end{proof}
\end{document}